\documentclass[11pt]{article}
\usepackage{amssymb}
\usepackage{amsfonts}
\usepackage{amsmath}
\usepackage{mathrsfs}
\usepackage{graphicx}
\usepackage{amsbsy}
\usepackage{theorem}
\usepackage{color}
\usepackage{hyperref}
\usepackage{tikz}
\usetikzlibrary{decorations.pathreplacing,calligraphy}
\usepackage[normalem]{ulem}
\usepackage{pgfplots}
\usepackage{subfig}
\usepackage{bbm}
\usepackage{multirow}
\newtheorem{theorem}{Theorem}[section]
\newtheorem{lemma}[theorem]{Lemma}
\newtheorem{corollary}[theorem]{Corollary}

\newtheorem{proposition}[theorem]{Proposition}
\newtheorem{remark}[theorem]{Remark}

\def\beq{\begin{equation}\displaystyle}
\def\eeq{\end{equation}}
\def\bel{\begin{equation} \displaystyle \begin{array}{l} }
\def\eel{\end{array} \end{equation} }
\def\bell{\begin{equation} \displaystyle \begin{array}{ll}  }
\def\eell{\end{array} \end{equation} }

\def\bea{\begin{eqnarray}}
\def\eea{\end{eqnarray} }
\def\bean{\begin{eqnarray*}}
\def\eean{\end{eqnarray*} }
\newenvironment{proof}{\noindent{\bf Proof.~}}
{{\mbox{}\hfill {\small \fbox{}}\\}}
\catcode`@=11
\def\NN{\mathbb{N}}

\def\RR{\mathbb{R}}

\def\bs{\bigskip}

\def\eps{\varepsilon}
\def\bar#1{{\overline #1}}

\def\tOm{\tilde{\Omega}_{\lambda^*}}

\definecolor{green}{rgb}{0,0.7,0}
\newcommand{\lrp}[1]{\left( #1 \right)}
\newcommand{\lrb}[1]{\left[ #1 \right]}
\newcommand{\lrk}[1]{\left\{ #1 \right\}}

\newcommand\sbullet[1][.7]{\mathbin{\vcenter{\hbox{\scalebox{#1}{$\bullet$}}}}}

\definecolor{GwenCorr}{rgb}{0.6,0,1}

\usepackage{ulem} % Pour barrer avec \sout{ }, etc.

\begin{document}

\title{Optimal strategies for \textit{Wolbachia} mosquito replacement technique: influence of the carrying capacity on spatial releases}

\author{Luis Almeida\footnote{Sorbonne Universit\'e, CNRS, Universit\'e de Paris, Inria, Laboratoire J.-L. Lions, 75005 Paris, France ({\tt luis.almeida@cnrs.fr}).}
	\and Jes\'us Bellver-Arnau\footnote{Sorbonne Universit\'e, CNRS, Universit\'e de Paris, Inria, Laboratoire J.-L. Lions, 75005 Paris, France ({Current address: Centre d’Estudis Avan\c{c}ats de Blanes (CEAB-CSIC), Carrer d’Accés a la cala Sant Francesc 14, 17300 Blanes, Spain, \tt jesus.bellver@ceab.csic.es}).}
	\and Gwenaël Peltier \footnote{Sorbonne Universit\'e, CNRS, Universit\'e de Paris, Inria, Laboratoire J.-L. Lions, 75005 Paris, France ({Current email: \tt gwenael.peltier1@gmail.com}).}
        %\and Yannick Privat\footnote{IRMA, Universit\'e de Strasbourg, CNRS UMR 7501, Inria, 7 rue Ren\'e Descartes, 67084 Strasbourg, France ({\tt yannick.privat@unistra.fr}).}\textsuperscript{~~}\footnote{Institut Universitaire de France (IUF).} 
	\and N. Vauchelet\thanks{Université Sorbonne Paris Nord, Laboratoire Analyse, Géométrie et Applications, LAGA, CNRS UMR 7539, F-93430, Villetaneuse, France. Email: \texttt{vauchelet@math.univ-paris13.fr}.}
}

\date{}

\maketitle

\begin{abstract}
	This work is devoted to the mathematical study of an optimization problem regarding control strategies of mosquito population in a heterogeneous environment. Mosquitoes are well-known vectors of diseases. For some diseases, such as dengue, it has been found that mosquitoes have a reduced vector capacity when carrying the endosymbiotic bacterium \textit{Wolbachia}. We consider a mathematical model of a replacement technique consisting in rearing and releasing \textit{Wolbachia}-infected mosquitoes to replace the wild population. Our goal is to optimize the release protocol to maximize replacement effectiveness in a spatially inhomogeneous environment. Using a scalar model with space-dependent carrying capacity, we explore the existence and properties of an optimal release profile maximizing the replacement across the domain. In particular, neglecting mosquito mobility and under some assumptions on the biological parameters, we characterize the optimal releasing strategy for a short time horizon, and we reduce the case of a long time horizon to a one-dimensional optimization problem. Our theoretical results are illustrated with several numerical simulations.
\end{abstract}

\bs



%{\bf 2010 AMS subject classifications: } 

\bs

\section{Introduction}\label{sec:Intro}

%\begin{color}{magenta}
	Mosquitoes of the genus \textit{Aedes} are responsible for the transmission of many diseases to humans. Among them, one may cite dengue, Zika, and chikungunya. One promising way to control the spread of such diseases is the use of the bacterium \textit{Wolbachia} \cite{Bourtzis}. Indeed, it has been reported that \textit{Wolbachia}-infected mosquitoes are less capable of transmitting such diseases \cite{Wal.wMel, Ant2018}.
	Moreover, this bacterium is characterized by, on the one hand, a vertical transmission from mother to offspring and, on the other hand, producing a \textit{cytoplasmic incompatibility}, which makes the mating of \textit{Wolbachia}-infected males with wild females to be less fertile \cite{Sinkins}.
	Taking advantage of these features, a population replacement strategy has been developed by rearing and releasing \textit{Wolbachia}-infected mosquitoes \cite{Kamtchum}. This strategy is becoming very popular since it is potentially self-sustaining and has been successfully deployed in the field by the World Mosquito Program (see \cite{Hoffmannetal, Deployment_Townsville} and \cite{WMP}).
	
	It has motivated several works including mathematical studies that have proved to be of great interest by proposing models that can be used to test and optimize different scenarios, see e.g. \cite{Li2018,Qu2018,ABPR} and the review article \cite{review}.
	In particular, mathematical control theory has been used to study the feasibility of the strategy, see e.g. \cite{2023agbobidiGlobal, Bliman, 2023Rossi}.
	A natural question is the optimization of the release protocol. How to design the best release protocol, taking into account production constraints, to have the most efficient population replacement?
	Several mathematical works have addressed this issue. First, the optimization of the temporal distribution of the releases, neglecting the spatial dependency, has been studied in e.g. \cite{colombien,APSV,MBE,ABP}. In particular, using a reduced model, the authors in \cite{APSV} show that the best strategy is close to a bang-bang strategy where the maximum number of mosquitoes available is released at one time at the beginning or end (depending of the number of  mosquitoes available) of the release period. 
	Then, the study of the optimization in space has been initiated in \cite{CEMRACS,DHP,Nadin_Toledo,Mazari_Nadin_Toledo} assuming that the space is homogeneous. In these works, the release is assumed to be done at initial time only and the question is to know how to spatially distribute the population to release in order to optimize the efficiency.
	
	However, it is clear that for practical applications, environment cannot be considered as homogeneous and heterogeneity may have a strong influence on the dynamics of the replacement (see e.g. \cite{NV_Blocking}). The aim of this paper is to propose a first mathematical study of an optimization problem for a population replacement in a non-homogeneous landscape. More precisely, we consider a two-population model inspired by the one studied in \cite{APSV}, see also \cite{Farkas,Fenton,Hughes,Colombien2}, where similar systems are considered. The key difference with \cite{APSV}, is that we assume that the carrying capacity of the environment is heterogeneous in space, modeling the fact that some areas are more favorable than others for the mosquito population.
	The system is first reduced in the spirit of \cite{SV_Reduction}, and the release is assumed to be point-wise in time: Indeed, it has been observed in \cite{APSV,ABPR} that this is often the best strategy.
	Moreover, as a first step towards understanding the effects of spatial heterogeneity in \textit{Wolbachia}-infected mosquito releases, we neglect the active motion of mosquitoes, usually modeled by a diffusion term. We comment in Section \ref{sec:Discussion} on the realism of these assumptions and the utility of the simplified model. To fix the notation, let us present briefly the resulting optimization problem considered in this article. Denote $p(t,x)$ the fraction of \textit{Wolbachia}-infected population at time $t>0$ and position $x\in \Omega$, with $\Omega$ being a bounded domain of $\RR$ or $\RR^2$.
	
	Let $u_0$ be a function of $x\in\Omega$, modeling a point-wise release of mosquitoes at initial time. We consider the control function $u_0$ to be positive, bounded by a constant $M$ and such that its integral (which corresponds to the total quantity of mosquitoes released) is bounded by $C$.

	 We seek to minimize the distance, at some given final time $T>0$, between the proportion $p$ and the total invasion constant steady state $1$. Using the $L^2$-norm, the problem reads
	\begin{equation}\label{prob:u0}\tag{$\mathcal{P}_{u_0}$}
		%\tag{\textrm{(OCP_{0})}}
		\boxed{
			\inf_{u_0\in\mathcal{U}_{0,C,M}} \int_\Omega K(x)^2 (1-p(T,x))^2\,dx,
		}
	\end{equation}
	with the space of admissible controls being
	\begin{equation}\label{def:U0CM}
		\mathcal{U}_{0,C,M} = \lrk{ u_0\in L^{\infty}(\Omega), \quad 0 \leq u_0\leq M \text{ a.e.}, \quad \int_\Omega u_0(x)\,dx \leq C}.
	\end{equation}

    In this problem, $K$ is the given carrying capacity depending on the space variable $x\in\Omega$. The relation between this release function, $u_0=u_0(x)$, and the proportion $p$ is obtained by solving a differential equation that will be specified later. Solving this optimal control problem boils down to investigating the question: what is the best initial distribution of the release in a heterogeneous environment in order to optimize the population replacement when active motion of mosquitoes is neglected?
	
	Up to our knowledge, this work is the first one considering spatial heterogeneity in such an optimization problem. We mention that neglecting active motion of mosquitoes allows us to reduce the problem to a dynamical system instead of considering a reaction-diffusion system, which greatly simplifies the study, and allows to derive precise results on the optimum. Although the study is simplified, it is nonetheless very challenging to get a precise description of the optimal solution. This work is a first step towards more sophisticated and more realistic future works where the active motion of mosquitoes will be considered.
	
	The outline of the paper is the following. In the next section, we present the modeling assumptions and introduce the optimization problem we are considering.
	Then, section \ref{sec:analysis} is devoted to the theoretical analysis of the optimization problem and to the proof of our main results. In section \ref{sec:numeric}, we illustrate optimal solutions by implementing an ad hoc numerical algorithm which exploits these results. Finally, some technical computations are provided in an appendix.

\section{Derivation of a mathematical model}

\subsection{The starting point: a two populations model with diffusion}
For the sake of completeness, we start by introducing a model based on two populations of mosquitoes, with and without \textit{Wolbachia}, which is the starting point to derive our model \eqref{eq:psimpl}. Let $u$ denote a function accounting for the rate at which \textit{Wolbachia}-infected mosquitoes are released in a given bounded domain, $\Omega$, of $\RR$ or $\RR^2$. To model the dynamics of the replacement of a wild population of mosquitoes by \textit{Wolbachia}-infected ones, a version of the following system has been considered in e.g. \cite{APSV}, although their model did not take space into account. In this system, the density of the \textit{Wolbachia}-infected species is denoted $n_2$, whereas the density of the wild population to be replaced is denoted $n_1$. The system reads
%\end{color}
\begin{equation}\label{sys:2species}
\begin{cases}
& \partial_t n_1 - D \Delta n_1 =b_1n_1\left(1-\displaystyle\frac{n_1+n_2}{K(x)}\right)\left(1-s_h\displaystyle\frac{n_2}{n_1+n_2}\right)-d_1n_1,\\
& \partial_t n_2 - D \Delta n_2 =b_2n_2\left(1-\displaystyle\frac{n_1+n_2}{K(x)}\right)-d_2n_2+u, \quad t\in[0,T], \quad x\in\Omega,  \\
&\partial_\nu n_1 = 0, \qquad \partial_\nu n_2 = 0, \quad \text{ on } \partial \Omega.
\end{cases}
\end{equation}
In this system, $b_i$, $i=1,2$, and $d_i$, $i=1,2$ represent the intrinsic birth and death rates of population $i$ (i.e. not considering the term representing the population limitation due to the carrying capacity). In many insect species, among which we find mosquitoes, the birth rate is considerably higher than the death rate, therefore for biological reasons we will consider the constraint $d_i\leq b_i$. Also, we assume that the second population has a fitness disadvantage with respect to the first one, due to the infection with \textit{Wolbachia}. This translates into imposing
\begin{equation}\label{fitness}
	b_2\leq b_1 \text{ and } d_1\leq d_2.
\end{equation}
The parameter $s_h$ measures the \textit{Wolbachia}-induced cytoplasmic incompatibility (CI) of the second population with respect to the first one. We have that $0\leq s_h \leq 1$, when $s_h=1$ there is a perfect CI, when $s_h=0$ there is none. We assume that, at $t=0$ there are no \textit{Wolbachia}-infected individuals and that the wild population is at equilibrium, that is $n_2(0,x)=0$ and $n_1(0,x)=K(x)\lrp{1-\frac{d_1}{b_1}}$. The diffusion coefficient is denoted $D$ and the carrying capacity $K$.
In \eqref{sys:2species}, $T>0$ is the time horizon of the problem. In this work we consider that the carrying capacity $K$ has Lipschitz spatial dependency:
$$
0<K(x) \in W^{1,\infty}(\Omega).
$$

The goal we pursue is to find an optimal release function, $u$, such that at a given final time $T$ the solution $(n_1,n_2)$ to \eqref{sys:2species} is as close as possible of the $n_2$-invasion steady state denoted $\lrp{0,n_2^*}=\lrp{0,K(x)\lrp{1-\frac{d_2}{b_2}}}$. Choosing a least square distance, this leads us to introduce the following cost functional
\begin{equation}\label{Ju}
J(u) = \frac 12 \int_\Omega \left( n_1(T,x)^2 + \lrb{(n_2^*-n_2(T,x))_+}^2\right)\,dx,
\end{equation}
where $(\cdot)_+$ stands for the positive part function, so that if $n_2(T,x)>n_2^*$ we have no penalty. We consider that both the number of available individuals to release and the rate at which these individuals can be released, are bounded from above.
The set of admissible controls is therefore given by
\begin{equation}\label{def.UTCM}
\mathcal{U}_{T,C,M} = \lrk{u\in L^{\infty}([0,T]\times\Omega), \quad 0\leq u\leq M \text{ a.e. }, \int_0^T\int_\Omega u(t,x)\,dtdx \leq C}.
\end{equation}

This is inspired by the problem considered in \cite{APSV}: we assume that there is a maximum release rate $M$ at each point in space and time and that there is a limited number of mosquitoes $C$ that can be used during the whole intervention (up to time $T$). Hence, we can state the following optimal control problem for System \eqref{sys:2species}:
\begin{equation}\label{prob:full}\tag{$\mathcal{P}_{2,\text{diff}}$}
	\inf_{u\in \mathcal{U}_{T,C,M}} J(u), \ \text{ where } \begin{cases}
		n_1 \text{ and } n_2 \text{ solve } \eqref{sys:2species}\\
		J \text{ is given by \eqref{Ju}} \\
		\mathcal{U}_{T,C,M} \text{ is given by \eqref{def.UTCM}}
	\end{cases} .
\end{equation}

This is a simplified setting that is used in this work but that will be extended to more natural settings in future works. For instance, it is natural to consider a constraint on the initial number of mosquitoes available and a limited time flux of the mosquitoes which would correspond to the situation where the totality of the mosquitoes are not immediately available but will be progressively produced by a mosquito production facility. In that case, for extra realism, it is also possible to consider the fact that mosquitoes that are stocked continue to get older and that this should be taken into account in their life expectancy once they are released.

\subsection{Model simplification}\label{sec:simplified_problem}

The problem as stated in \eqref{prob:full} is a very challenging one, therefore, we make three simplifying assumptions to make the problem more treatable. 
	
\begin{paragraph}{High fertility.}
The arguments found in \cite{SV_Reduction} and \cite[Proposition 2.2]{APSV} can be adapted to prove that when fecundity rates are large, that is, if we assume that $b_1 = \frac{b_1^0}{\eps}$ and $b_2 = \frac{b_2^0}{\eps}$ and we let $\eps\to 0$, system \eqref{sys:2species} may be reduced to one single equation on the proportion of \textit{Wolbachia}-infected mosquitoes in the population. The addition of an inhomogeneous carrying capacity does not alter the main arguments of the proof of this reduction and we may obtain after straightforward computations that $p=p(t,x)$ solves the following scalar equation:
\begin{equation}\label{eq:p}
	\left\{
	\begin{array}{l}
		\displaystyle \partial_t p - D \Delta p - 2 D \frac{\nabla p\cdot \nabla K(x)}{K(x)} = f(p)+ \frac{u(t,x)}{K(x)} g(p) - D \frac{\Delta K(x)}{K(x)} \psi(p), \quad t>0, \quad x\in\Omega \\
		p(0,x) = 0, \quad x\in\Omega,\\
		\partial_\nu p(t,x) = 0, \quad t>0, \  x \in \partial \Omega;
	\end{array}
	\right.
\end{equation}
where 
\begin{equation}\label{def:fg}
	f(p)= b_1^0 d_2 s_h \frac{p(1-p)(p-\theta)}{b_1^0 (1-p) (1-s_h p)+ b_2^0 p}\quad \text{and}\quad g(p)=\frac{b_1^0 (1-p) (1-s_h p)}{b_1^0 (1-p) (1-s_h p)+b_2^0 p},
\end{equation}
%\begin{color}{magenta}
and
\begin{equation}\label{def:psi}
	\psi(p) = \frac{p(1-p)(b_2^0 -b_1^0(1-s_h p))}{(1-p) (1-s_h p)+b_2^0 p},
\end{equation}
%\end{color}
with
\begin{equation}\label{theta}
	\theta = \frac{1}{s_h}\lrp{1-\frac{d_1 b_2^0}{d_2 b_1^0}},
\end{equation}
which is strictly bounded between 0 and 1 under the condition $1-s_h<\frac{d_1 b_2^0}{d_2 b_1^0}<1$, which is a consequence of assumption \eqref{fitness}. This condition ensures that $f$ has  a root in $(0,1)$, which is biologically plausible, and will be an important feature in the analysis of the problem studied in this work.
Moreover, the cost functional $J$ reduces to
\begin{equation}\label{J0}
	J^0(u) = \int_\Omega K(x)^2(1-p(T,x))^2\,dx,
\end{equation}
and we obtain an asymptotic version of Problem \eqref{prob:full} reading
\begin{equation}\label{prob:reduced}\tag{$\mathcal{P}_{1,\text{diff}}$}
	\inf_{u\in \mathcal{U}_{T,C,M}} \int_\Omega K(x)^2 (1-p(T,x))^2\,dx, \ \text{ where } \begin{cases}
	p \text{ solves \eqref{eq:p}} \\
	\mathcal{U}_{T,C,M} \text{ is given by \eqref{def.UTCM}}
	\end{cases} .
\end{equation}
This model reduction allows us to study the problem in simpler terms, knowing that solutions of the simplified Problem \eqref{prob:reduced} will be asymptotically close to solutions of Problem \eqref{prob:full} in the sense of Gamma-convergence (we refer to \cite{DHP,SV_Reduction} for details). This simplification is consistent from a biological point of view as, in temperature ranges fit for mosquito life, birth rates can be up to three orders of magnitude higher than adult death rates for several species of disease-transmitting mosquitoes \cite{Mordecai2019}.
\end{paragraph}

\begin{paragraph}{Negligible mobility and initial release}
We further assume that the time distribution of the release is given by $u(t,x) = u_0(x) \delta_0(t)$. In other words, we consider that there is one single release, done at the initial time, and that the time it takes to do the release is negligible in comparison with the time window considered. Following the reasoning developed in \cite{DHP}, and, as stated in Section \ref{sec:Intro}, neglecting the active motion of mosquitoes by considering that the diffusion coefficient $D=0$, one can prove that equation \eqref{eq:p} simplifies into
\begin{equation}
	\label{eq:psimpl}
	\begin{cases}
		& \partial_t p(t,x) = f(p(t,x)), \quad t\in[0,T], \quad x\in\Omega, \\
		& p(0^+,x) = G^{-1}\lrp{\frac{u_0(x)}{K(x)}}, \quad x\in\Omega,
	\end{cases}
\end{equation}
where the function $G$ is defined as the antiderivative vanishing at zero of $1/g$, i.e.
$$
G(p) = \int_0^p \frac{d\nu}{g(\nu)}.
$$
Now that we performed the simplification of equation \eqref{eq:p} into \eqref{eq:psimpl}, the optimization Problem \eqref{prob:reduced} is recast as the Problem \eqref{prob:u0} presented in Section \ref{sec:Intro}, with $u_0$ linked to the initial data of \eqref{eq:psimpl}. However, we shall perform a last transformation of the Problem \eqref{prob:u0} into another equivalent one.
\end{paragraph}

\subsection{Optimal control problem studied}

Because $g$ is positive, we see that $G$, thus $G^{-1}$, is increasing. Looking at \eqref{eq:psimpl}, this implies that there is a one-to-one relation between the release carried at the initial time $u_0(x)$ and the initial data $p(0^+,x)$. Therefore, we can reformulate Problem \eqref{prob:u0} in terms of this initial proportion by defining $p_0(x):=G^{-1}\lrp{\frac{u_0(x)}{K(x)}}$ and considering it the new control variable. This leads to the following problem which shall be our focus in the present paper :

\begin{equation}
	\label{eq:psimpl_p0}
	\begin{cases}
		& \partial_t p(t,x) = f(p(t,x)), \quad t\in[0,T], \quad x\in\Omega, \\
		& p(0^+,x) = p_0(x), \quad x\in\Omega,
	\end{cases}
\end{equation}
With this new control variable, Problem \eqref{prob:u0} is rewritten as :
\begin{equation}\label{prob:p0}\tag{$\mathcal{P}_{p_0}$}
  \boxed{
    \inf_{p_0\in\mathcal{P}_{0,C,M}} \int_\Omega K(x)^2 (1-p(T,x))^2\,dx
  } \ \text{ where } \begin{cases}
	p \text{ solves \eqref{eq:psimpl_p0}} \\
	\mathcal{P}_{0,C,M} \text{ is given by \eqref{def:P0CM}}
	\end{cases} .
\end{equation}
\begin{equation}\label{def:P0CM}
%\resizebox{0.9\textwidth}{!}{\boxed{
  \mathcal{P}_{0,C,M} = \lrk{ p_0\in L^{\infty}(\Omega), \quad 0 \leq p_0\leq G^{-1}\lrp{\frac{M}{K(x)}} \text{ a.e.}, \quad \int_\Omega K(x) G(p_0(x))\,dx \leq C}.
%}}
\end{equation}

It should be noted that, since $g$ is decreasing, we have that $G$ is convex. Also, the underlying cost functional $J^0$ defined in \eqref{J0} is clearly continuous. However, because of the term $K(x)$, the cost functional $J^0$ is not necessarily convex (unless we add some very restrictive assumptions on $f$). On the one hand, Problem \eqref{prob:u0} involves only ordinary differential equation contrary to Problem \eqref{prob:full}. On the other hand, the lack of convexity highlighted above makes the analysis of Problem \eqref{prob:u0} quite challenging.

\section{Analysis of the optimal problem}\label{sec:analysis}
\subsection{Preliminaries}
Before analysing the problem in depth, we can obtain easily the following lemmas that will be useful later for the characterization of the solutions. First, we observe that the constraint of the problem is saturated and that the case $M|\Omega|\leq C$ is trivial:
\begin{lemma} \label{lem:constr_sat}
  If $u_0^*=K(x)G(p_0^*)$ is an optimal solution of the optimal control Problem \eqref{prob:u0} and $M|\Omega|>C$, then $\displaystyle \int_\Omega u_0^*(x)\,dx = C$, or, equivalently, $\displaystyle \int_\Omega K(x) G(p_0^*(x))\,dx = C$.
  
  Also, if $M|\Omega| \leq C$, the optimal solution is given by $u_0^* = M$ or equivalently $p_0^* = p_M := G^{-1}(\frac{M}{K})$.
\end{lemma}
\begin{proof}
  The first point is a trivial consequence of the fact that, on the one hand, $G$ is increasing therefore so is $G^{-1}$, and on the other hand, the solutions of \eqref{eq:psimpl} are ordered, that is if $p_1(0^+,x)\leq p_2(0^+,x)$ then $p_1(\cdot,x)\leq p_2(\cdot,x)$.

  Using also the monotonicity of the solutions of \eqref{eq:psimpl} with respect to their initial data, we get the second point.
\end{proof}

\begin{remark}\normalfont{
    We recall that, problems \eqref{prob:p0} and \eqref{prob:u0} are equivalent. Indeed, since $G$ is continuous and strictly increasing, we have the simple one-to-one relation $u_0^*(x):=K(x)G(p_0^*(x)).$
	}
\end{remark}

\subsection{Optimality conditions}

As a consequence of the previous lemma, we will always assume, from now on, that $M|\Omega| > C$. The following Lemma provides a description of the optimal solution using the optimality conditions:
\begin{lemma}\label{lem:w}
  Let us define the switch function $w_T$ by
  \begin{equation}
    \label{eq:w}
    \fbox{$\displaystyle w_{T}(p_{0}) =- g(p_0) (1-p(T,x)) \exp\left(\int_0^T f'(p(s,x))\,ds\right) < 0$}.
  \end{equation}
  Let $u_0^*$ be any optimal solution of the optimal control problem \eqref{prob:u0}.
  Then, there exists $\lambda^* \geq 0$ such that the optimal solution $u_0^*$ verifies the first order conditions:
\begin{itemize}
\item on $\{u_0^* = M\}=\lrk{p_0^* = p_M := G^{-1}\lrp{\frac{M}{K(x)}}}$, we have $w_T\leq -\frac{\lambda^*}{K(x)}$,
\item on $\{u_0^*=0\}=\{p_0^* = 0\}$, we have $w_T\geq -\frac{\lambda^*}{K(x)}$,
\item on $\{0<u_0^*<M\}=\{0<p_0^*<p_M\}$,  we have $w_T=-\frac{\lambda^*}{K(x)}$ and each minimum should verify the second order condition
$$
\frac{\partial w_T}{\partial p_0} \geq 0.
$$
\end{itemize}
\end{lemma}
\begin{remark}\normalfont{
It will be useful in the following to notice that in \eqref{eq:w}, the switch function depends on $x$ only through the initial condition $p_0(x)$.}
\end{remark}

\begin{proof} We verify the first and second order optimality conditions.
\begin{itemize}
\item \textbf{First order optimality condition}
\end{itemize}

Let us introduce the Lagrangian
$$
\mathcal{L}(p_0) = \frac 12\int_\Omega K(x)^2(1-p(T,x))^2\,dx + \lambda\left(\int_\Omega K(x) G(p_0)\,dx - C\right),
$$
for some $\lambda\in\RR^+$.
To compute its derivative, we introduce the linearized system
\begin{equation}
  \label{eq:linear_simpl}
  \partial_t \delta p = f'(p) \delta p, \qquad \delta p(0^+,x) = h,
\end{equation}
and the adjoint state
\begin{equation}
  \label{eq:adjoint_simpl}
  -\partial_t q = f'(p) q, \qquad q(T,x) = -K(x)^2 (1-p(T,x)) < 0.
\end{equation}
In particular, from \eqref{eq:linear_simpl} we deduce 
\begin{equation}
\frac{\partial p(T,x)}{\partial p_0(x)} = \exp\left(\int_0^T f'(p(s,x))ds \right).\label{eq:dpT/dp0}
\end{equation}
Then, to verify the first order optimality condition, we compute
$$
d\mathcal{L}(p_0)\cdot h = - \int_\Omega K(x)^2 (1-p(T,x)) \delta p(T,x)\,dx + \lambda\left(\int_\Omega K(x) G'(p_0) h \,dx\right).
$$
Using \eqref{eq:linear_simpl} and \eqref{eq:adjoint_simpl}, we deduce
$$
0 = \int_0^T \int_\Omega \partial_t (\delta p q) \,dxdt = \int_\Omega \delta p(T,x) q(T,x) \,dx - \int_\Omega h q(0,x)\,dx.
$$
Therefore,
\begin{equation}
  \label{eq:calL}
d\mathcal{L}(p_0)\cdot h = \int_\Omega h \left( q(0,x) + \lambda K(x) G'(p_0)\right)\,dx
= \int_\Omega \frac{K(x)}{g(p_0)} h \left( \frac{1}{K(x)} g(p_0) q(0,x) + \lambda \right)\,dx,
\end{equation}
where we have used the fact that $G'=\frac{1}{g}$.
Then, solving \eqref{eq:adjoint_simpl} we get
$$
q(0,x) = - K(x)^2 (1-p(T,x)) \exp\left(\int_0^T f'(p(s,x))\,ds\right).
$$
Injecting into \eqref{eq:calL}, we obtain
\begin{equation}
  \label{eq:calL2}
  d\mathcal{L}(p_0)\cdot h = \int_\Omega \frac{K(x)^2}{g(p_0(x))} h(x) \left( w_T(p_0(x)) + \frac{\lambda}{K(x)}\right)\,dx.
%  \omega_{x,T}(p_{0}) := - K(x) g(p_0(x)) (1-p(T,x)) \exp\left(\int_0^T f'(p(s,x))\,ds\right) < 0.
\end{equation}
% Note that function $\omega_{x,T}$ only depends on $x$ through $K(x)$ and the initial condition $p_0(x)$, therefore, when looked as a function of the initial condition, the only dependency of $p_0\mapsto\omega_{x,T}$ on $x$ is through $K(x)$. It will be useful in the following to work with a different switch function that will allow us to simplify the exposition by rendering the switching function independent of $x$. This will simplify the characterization of the solutions in Theorems \ref{theo:TleqT0} and \ref{theo:TgeqT0}. We introduce
% \begin{equation}
% \label{eq:w}
% \fbox{$\displaystyle w_{T}(p_{0}) := \frac{\omega_{x,T}(p_0)}{K(x)} =- g(p_0(x)) (1-p(T,x)) \exp\left(\int_0^T f'(p(s,x))\,ds\right) < 0$}.
% \end{equation}

% We may also compute some particular values of the switch function :
% \begin{itemize}
% \item In $\{p_0 = 0\}$, i.e. on the set $\{u_0=0\}$, we have $p(t,x)=0$ for all $t\geq 0$ and we get $w_0 = - g(0) \exp(T f'(0))$.
% \item In $\{p_0= \theta\}$, $p(t.x)=\theta$ for all $t\geq 0$, and
%   $w_\theta = - g(\theta)(1-\theta) \exp( T f'(\theta))$.
% \end{itemize}

Then, noting that the function $g$ is positive, thanks to a classical application of the Pontryagin Maximum Principle (PMP) \cite{Trelat_Book}, we obtain that there exists $\lambda^*>0$ such that, on the set $\{u_0^* = 0\}$, we have $w_T+\frac{\lambda^*}{K(x)} \geq 0$, on the set $\{u_0^* = M\}$, we have $w_T+\frac{\lambda^*}{K(x)} \leq 0$, and on the set $\{0<u_0^*<M\}$, we have $w_T+\frac{\lambda^*}{K(x)} = 0$.

\begin{itemize}
\item \textbf{Second order optimality condition}
\end{itemize}

We compute the second order derivative of the Lagrangian from the expression \eqref{eq:calL2},
$$
d^2\mathcal{L}(p_0)\cdot h\cdot h = - \int_\Omega \frac{K(x)^2 g'(p_0)}{g(p_0)^2} h^2\lrp{w_T + \frac{\lambda}{K(x)}}\,dx + \int_\Omega \frac{K(x)^2}{g(p_0)} h^2 \frac{\partial w_T}{\partial p_0} \,dx.
$$
Then, on the set $\{0<u_0^*<M\}=\{0<p_0^*<p_M\}$, we have that
$$
d^2\mathcal{L}(p_0)\cdot h\cdot h = \int_\Omega \frac{K(x)^2}{g(p_0)} h^2 \frac{\partial w_T}{\partial p_0} \,dx.
$$
At the minimum we must have $d^2\mathcal{L}(p_0)\cdot h\cdot h\geq 0$ for every $h$. Since $g$ is positive, the result follows.
\end{proof}

% \begin{lemma}\label{lem:secondorder}
% On the set $\{0<u_0^*<M\}=\{0<p_0^*<p_M\}$, each minimum should verify the condition
% $$
% \frac{\partial w_T}{\partial p_0} \geq 0.
% $$
% \end{lemma}

\subsection{Study of the switch function}

We devote this section to the study of the switch function $w_{T}(p_{0})$ defined in \eqref{eq:w}. This function depends on $T$ and the initial condition $p_0(x)$. Nevertheless, in this section, we are only interested in its behavior as a function of the initial condition. Therefore, in what follows, we fix a $x\in\Omega$ and we consider $T$ as a parameter. To simplify the notation, we just write $w(p_0)$. Thus, 
$$w(p_0)=-g(p_0) (1-p(T)) \exp\left(\int_0^T f'(p(s))\,ds\right)$$ 

We now present a Lemma on the monotonicity of $w$ that will play a crucial role in the characterization of the solutions of Problem \eqref{prob:p0}. Before stating it, we require some additional preliminaries and notations. Since we assumed $b_{2}^{0}\leq b_{1}^{0}$ and $d_{1}\leq d_{2}$,
one can prove (see Appendix \ref{app:H}) that $f^{\prime\prime}$
admits a unique zero $\theta_{2}$ in $(0,1)$. Additionally, for
any $p\in[0,1]$, we have $f^{\prime\prime}(p)>0$ if and only if
$p<\theta_{2}$. Setting $\overline{\theta}:=\max(\theta,\theta_{2})$, with $\theta$  defined in \eqref{theta}, we now introduce the following function
\begin{equation}\label{def:A}
p_0\mapsto A(p_0):=\frac{g'(p_0)}{g(p_0)}-\frac{1}{1-p(T)}e^{\int_0^Tf'(p(s))ds}+\int_0^Tf''(p(s))e^{\int_0^s f'(p(\sigma))d\sigma}ds.
\end{equation}
and the following hypothesis on it
\begin{equation} \tag{$\mathcal{H}$} \label{H}
\text{Function } p_0\mapsto A(p_0) \text{ changes sign at  most once in } (0,\bar{\theta}).
\end{equation}
We investigate
for which parameters \eqref{H} is true in Appendix \ref{app:H}.

\begin{lemma}\label{lem:w_monotonicity}
Assume \eqref{H} holds. There exists $T_0>0$, given by
\begin{equation} \label{eq:T0}
	T_0:=\frac{1}{f'(0)}\ln\lrp{\frac{ f''(0)g(0)-f'(0)g'(0)}{g(0)\lrp{f''(0)-f'(0)}}},
\end{equation}
such that, if $T\leq T_0$, then $\frac{\partial w}{\partial p_0} (p_0)> 0$ for any $p_0\in(0,1)$. If $T> T_0$, there exists one single $p_0^{T}\in(0,\bar{\theta})$ such that 
\begin{equation}
\frac{\partial w}{\partial p_0}(p_0) < 0, \quad \forall p_0\in(0,p_0^T) \quad \textup{and} \quad \frac{\partial w}{\partial p_0}(p_0) > 0, \quad \forall p_0\in(p_0^T,1).\label{eq:p0T}
\end{equation}
\end{lemma}
Let us notice that we obtain an explicit expression of $T_0$ which is given in \eqref{eq:T0} in the proof below.

\begin{proof} We first look at the sign of $\frac{\partial w}{\partial p_0}$ at $p_0=0$ to derive the value of $T_0$.
\begin{paragraph}{Sign of $\frac{\partial w}{\partial p_0}$ at $p_0=0$.}
Recalling \eqref{eq:dpT/dp0}, a simple computation yields
\begin{align} \nonumber 
\frac{\partial w}{\partial p_0}  = & -g'(p_0)(1-p(T))e^{\int_0^Tf'(p(s))ds)}+g(p_0)e^{2\int_0^Tf'(p(s))ds} \\ \label{eq:part_w}
 & - g(p_0)(1-p(T))e^{\int_0^T f'(p(s))ds}\int_0^Tf''(p(s))e^{\int_0^s f'(p(\sigma))d\sigma}ds,
\end{align}
As a result, using that when $p_0=0$ we have that $p(t)=0$ for all $t\in[0,T]$, we obtain
 \begin{align*}
\frac{\partial w}{\partial p_0}\bigg|_{p_0=0}  = & -g'(0)e^{T f'(0)}+g(0)e^{2 T f'(0)} - g(0)e^{T f'(0)}\int_0^Tf''(p(s)))e^{s f'(0)}ds\Big] \\
= & -e^{T f'(0)}\Big[g'(0)-g(0)e^{T f'(0)} + g(0)\frac{f''(0)}{f'(0)}\lrp{e^{Tf'(0)}-1}\Big].
\end{align*}Let us recall that $g(0)>0>g'(0)$ and $f'(0)<0<f''(0)$. From the above expression we deduce that

\begin{align*}
& \frac{\partial w}{\partial p_0}\bigg|_{p_0=0}\geq 0 \\
\Leftrightarrow & \ g'(0)-g(0)e^{T f'(0)} + g(0)\frac{f''(0)}{f'(0)}\lrp{e^{Tf'(0)}-1} \leq 0  \\ 
\Leftrightarrow & \  e^{T f'(0)}g(0)\lrp{\frac{f''(0)}{f'(0)}-1}\leq g(0)\frac{f''(0)}{f'(0)}-g'(0)  \\ 
 \Leftrightarrow & \ e^{T f'(0)}\geq \frac{ f''(0)g(0)-f'(0)g'(0)}{g(0)\lrp{f''(0)-f'(0)}}  \\
 \Leftrightarrow & \ T \leq \frac{1}{f'(0)}\ln\lrp{\frac{ f''(0)g(0)-f'(0)g'(0)}{g(0)\lrp{f''(0)-f'(0)}}} \\
  \Leftrightarrow & \ T \leq T_0.
\end{align*}
One can check that the value of $T_0$ is always well defined and positive. Indeed, the argument of the logarithm is positive since $$f''(0)g(0)-f'(0)g'(0)=\frac{b_2^0}{(b_1^0)^2}(2d_1b_1^0s_h + d_2b_1^0 - b_2^0d_1)>\frac{b_2^0}{(b_1^0)^2}2d_1b_1^0s_h>0.$$ We recall that $d_2b_1^0>b_2^0d_1$ by the biological assumptions made so far. On the other hand, the argument of the logarithm is also smaller than one, since $g(0)=1$ and $f''(0)-f'(0)g'(0)<f''(0)-f'(0)$. Consequently, if $T>T_0$, then $\frac{\partial w}{\partial p_0} < 0 $ in a neighborhood of $p_0=0$. 
\end{paragraph}

Note that we can rewrite expression \eqref{eq:part_w} using $A(p_0)$ as defined in equation \eqref{def:A},
\begin{eqnarray}
\nonumber \frac{\partial w}{\partial p_0} & = & -g(p_0)(1-p(T))e^{\int_0^Tf'(p(s))ds)}\lrb{\frac{g'(p_0)}{g(p_0)}-\frac{1}{1-p(T)}e^{\int_0^Tf'(p(s))ds}+\int_0^Tf''(p(s))e^{\int_0^s f'(p(\sigma))d\sigma}ds}\\ \label{eq:part_w_var}
   & = & \ w(p_0) A(p_0),
\end{eqnarray}
Recall that $w(p_0)<0$ and therefore $\frac{\partial w}{\partial p_0}$ changes signs as many times, and at the same points, as $A(p_0)$.

%Differentiating again \eqref{eq:part_w_var} we obtain $\frac{\partial^2 w}{\partial p_0^2}  =  \frac{\partial w}{\partial p_0}A(p_0)+w(p_0)\frac{\partial A}{\partial p_0}$. We compute:
%\begin{align} \nonumber 
%\frac{\partial A}{\partial p_0}  = &  \ \frac{g''(p_0)g(p_0)-g'(p_0)^2}{g(p_0)^2}-\frac{1}{\lrp{1-p(T)}^2}e^{2\int_0^Tf'(p(s))ds} \\ \nonumber
% & -\frac{1}{1-p(T)}e^{\int_0^Tf'(p(s))ds}\int_0^Tf''(p(s))e^{\int_0^s f'(p(\sigma))d\sigma}ds +\int_0^T f^{(3)}(p(s))e^{2\int_0^s f'(p(\sigma))d\sigma} \ ds \\ \label{eq:part_A} 
% & +\int_0^T f''(p(s))e^{\int_0^s f'(p(\sigma))d\sigma}\int_0^s f''(p(\sigma))e^{\int_0^\sigma f'(p(\tau))d\tau} \ d\sigma \ ds .
%\end{align}

Fix $T>0$ and set $\bar{\theta}:=\max(\theta,\theta_2)\in(0,1)$. Let us prove $\frac{\partial w}{\partial p_0}>0$ for all $p_0\in(\bar{\theta},1)$. From \eqref{eq:part_w_var} and \eqref{eq:w}, it is enough to prove that $A=A(p_0)$ defined by \eqref{def:A} is negative on $(\bar{\theta},1)$. The first two terms of \eqref{def:A} are strictly negative for all $p_0\in(0,1)$. Therefore it is sufficient to prove that 
$$\text{If }p_0\in (\bar{\theta},1),\qquad f''(p(t))\leq 0  \mbox{ for all } t\in(0,T).$$
Let us recall that $\theta_2$ is the unique zero of $f''$ in $(0,1)$, and that $f''<0$ in $(\theta_2,1)$ (see Proposition \ref{prop:f2zero}). Let $p_0>\bar{\theta}$. Then, since $p_0>\theta$, one can quickly check that $t\mapsto p(t)$ is non-decreasing. Therefore, $p(t)\geq p_0 > \theta_2$, so that $f''(p(t))\leq 0$ for all $t$. Thus $\frac{\partial w}{\partial p_0}>0$ on $(\bar{\theta},1)$.

\begin{paragraph}{Conclusion.}
In conclusion, if $T\leq T_0$, we proved that $\frac{\partial w}{\partial p_0}\big|_{p_0=0}$ is positive and so is $\frac{\partial w}{\partial p_0}$ for all $p_0\in(\bar{\theta},1)$. By Hypothesis \eqref{H}, $\frac{\partial w}{\partial p_0}$ changes sign at most once, by contradiction $\frac{\partial w}{\partial p_0}$ cannot change sign in $(0,\bar{\theta})$, and thus $\frac{\partial w}{\partial p_0}\geq 0$ for all $p_0\in(0,1)$. On the other hand, if $T>T_0$, $\frac{\partial w}{\partial p_0}\big|_{p_0=0} < 0$ and $\frac{\partial w}{\partial p_0}>0$ for all $p_0>\bar{\theta}$. Therefore $w$ has at least one minimum. Again, by Hypothesis \eqref{H}, $\frac{\partial w}{\partial p_0}$ changes sign at most once, and thus the minimum, that we note $p_0^T$, must be unique. \eqref{eq:p0T} follows.
\end{paragraph}

\end{proof}

\begin{figure}[htbp!]
	\centering 
	\begin{tikzpicture}[brace/.style={thick,decorate,
		decoration={calligraphic brace, amplitude=7pt,raise=0.5ex}},scale=1.3]
	
	%Horizontal axis
	\draw[->] (0,5) -- (5.5,5)node[anchor=west] {$p_0$};
	
	%Vertical axis
	\draw[->, color=purple] (5,0) -- (5,5.5); 
	%\draw (0,4.5) node[anchor=east] {$w(p_0)$}; 

	%Plot	
	
	\draw [domain=0:4.2, line width=1.0pt, color=blue] plot ({\x},{2.5+rad(atan(2*(\x-2)))});

	%Labels
	%\draw[blue,fill=blue] (0,2.5-1.326) circle (.25ex);
	\draw[blue,fill=blue] (2,2.5) circle (.25ex);
	%\draw[blue,fill=blue] (4.2,2.5+1.347) circle (.25ex);
	
	\draw[blue,fill=blue] (0,5) circle (.25ex);
	\draw[blue,fill=blue] (2,5) circle (.25ex);
	\draw[blue,fill=blue] (4.2,5) circle (.25ex);
	
	\draw[blue,fill=blue] (5,0.5) circle (.25ex);
	\draw[blue,fill=blue] (5,2.5) circle (.25ex);
	\draw[blue,fill=blue] (5,4.4) circle (.25ex);
	
	\draw [dotted, line width=1.0pt, color=green] (4.2,4.4) -- (5,4.4)node[anchor=west,color=purple]{$-\lambda^*/K(x)$};
	\draw [dotted, line width=1.0pt, color=green] (4.2,4.4) -- (4.2,5)node[anchor=south,color=black]{$p_0^*=p_M$};
	
	\draw [dotted, line width=1.0pt, color=green] (2,2.5) -- (5,2.5 )node[anchor=west,color=purple]{$-\lambda^*/K(x)$};
	\draw [dotted, line width=1.0pt, color=green] (2,2.5) -- (2,5)node[anchor=south,color=black]{$p_0^*=w_T^{-1}\lrp{-\frac{\lambda^*}{K(x)}}$};
	
	\draw [dotted, line width=1.0pt, color=green] (0,0.5) -- (5,0.5 )node[anchor=west,color=purple]{$-\lambda^*/K(x)$};
	\draw [dotted, line width=1.0pt, color=green] (0,0.5) -- (0,5)node[anchor=south,color=black]{$p_0^*=0$};

	\draw [dashed, line width=1.0pt, color=black] (0,2.5+1.347) -- (5,2.5+1.347)node[anchor=west]{$w_T(p_M)$}; 
	\draw [dashed, line width=1.0pt, color=black] (0,2.5-1.326) -- (5,2.5-1.326)node[anchor=west]{$w_T(0)$};

	\end{tikzpicture}
	\caption{Typical shape of $p_0\mapsto w_{T}(p_0)$, in the case $T\leq T_0$.}\label{fig:TleqT0}
\end{figure}

\begin{figure}[htbp!]
	\centering 
	\begin{tikzpicture}[brace/.style={thick,decorate,
		decoration={calligraphic brace, amplitude=7pt,raise=0.5ex}},scale=1.2]
	
	%Plot 1
	
	%Horizontal axis
	\draw[->] (0,5) -- (5.5,5)node[anchor=west] {$p_0$};
	
	%Vertical axis
	\draw[->, color=purple] (5,0) -- (5,5.5); 
	%\draw (0,4.5) node[anchor=east] {$w(p_0)$}; 

	%Plot	
	\draw [domain=0:1, line width=1.0pt, color=blue] plot ({\x},{0.8+1.8*(\x^2-2*\x*1.1+1.1^2)});
	
	%Labels
	
	\draw [dashed, line width=1.0pt, color=black] (0,0.8) -- (5,0.8)node[anchor=west] {$w(p_M)$};
	\draw [dashed, line width=1.0pt, color=black] (0,0.8+1.8*1.1^2) -- (5,0.8+1.8*1.1^2) node[anchor=west] {$w(0)$};
	
	\draw [dotted, line width=1.0pt, color=green] (1,3.7) -- (5,3.7)node[anchor=west,color=purple]{$-\lambda^*/K(x)$};
	\draw [dotted, line width=1.0pt, color=green] (1,3.7) -- (1,5)node[anchor=south,color=black]{$\qquad p_0^*=p_M$};
	
	\draw [dotted, line width=1.0pt, color=green] (0.6,1.682) -- (5,1.682)node[anchor=west,color=purple]{$-\lambda^*/K(x)$};
	\draw [dotted, line width=1.0pt, color=green] (0.6,1.682) -- (0.6,5.5)node[anchor=south,color=black]{$p_0^*\in\left\{ 0,p_M \right\}$};
	
	\draw [dotted, line width=1.0pt, color=green] (0,0.3) -- (5,0.3)node[anchor=west,color=purple]{$-\lambda^*/K(x)$};
	\draw [dotted, line width=1.0pt, color=green] (0,0.3) -- (0,5)node[anchor=south,color=black]{$p_0^*=0 \quad$}node[anchor=north east,color=black]{\Large{A)} \qquad};
	
	%\draw[blue,fill=blue] (1,0.818) circle (.25ex);
	\draw[blue,fill=blue] (1,5) circle (.25ex);
	\draw[blue,fill=blue] (5,0.3) circle (.25ex);
	
	%\draw[blue,fill=blue] (0.4,1.682) circle (.25ex);
	%\draw[blue,fill=blue] (0.6,5) circle (.25ex);
	\draw[blue,fill=blue] (5,1.682) circle (.25ex);
	
	\draw[blue,fill=blue] (5,3.7) circle (.25ex);
	%\draw[blue,fill=blue] (0,2.978) circle (.25ex);
	\draw[blue,fill=blue] (0,5) circle (.25ex);
	
	%Plot 2
	%Horizontal axis
	\draw[->] (0,5-6.2) -- (5.5,5-6.2)node[anchor=west] {$p_0$};
	
	%Vertical axis
	\draw[->, color=purple] (5,0-6.2) -- (5,5.5-6.2); 
	%\draw (0,4.5-6.2) node[anchor=east] {$w(p_0)$}; 
	
	%Plot	
	\draw [domain=0:1.9, line width=1.0pt, color=blue] plot ({\x},{-6.2+0.8+1.8*(\x^2-2*\x*1.1+1.1^2)});
	
	%Labels
	
	\draw [dashed, line width=1.0pt, color=black] (0,-6.2+0.8+1.8*1.1^2) -- (5,-6.2+0.8+1.8*1.1^2)node[anchor=west] {$w(0)$};
	\draw [dashed, line width=1.0pt, color=black]  (0,-6.2+0.8+1.8*0.8^2) -- (5,-6.2+0.8+1.8*0.8^2) node[anchor=west] {$w(p_M)$};
	\draw [dashed, line width=1.0pt, color=black] (0,-6.2+0.8) -- (5,-6.2+0.8) node[anchor=west] {$\min_{p_0} w(p_0)$};
	
	\draw [dotted, line width=1.0pt, color=green] (1.9,3.7-6.2) -- (5,3.7-6.2)node[anchor=west,color=purple]{$-\lambda^*/K(x)$};
	\draw [dotted, line width=1.0pt, color=green] (1.9,3.7-6.2) -- (1.9,4.4-6.2)node[anchor=south,color=black]{$\qquad \quad p_0^*=p_M$};
	
	\draw [dotted, line width=1.0pt, color=green] (0.35,2.425-6.2) -- (5,2.425-6.2)node[anchor=west,color=purple]{$-\lambda^*/K(x)$};
	\draw [dotted, line width=1.0pt, color=green] (0.35,2.425-6.2) -- (0.35,5.5-6.2)node[anchor=south,color=black]{$p_0^*\in\left\{0,p_M\right\} \qquad$};
	
	\draw [dotted, line width=1.0pt, color=green] (1.6,1.25-6.2) -- (5,1.25-6.2)node[anchor=west,color=purple]{$-\lambda^*/K(x)$};
	\draw [dotted, line width=1.0pt, color=green] (1.6,1.25-6.2) -- (1.6,5-6.2)node[anchor=south,color=black]{$\qquad \qquad p_0^*\in\left\{0,w^{-1}\lrp{-\frac{\lambda^*}{K(x)}}\right\}$};
	
	\draw [dotted, line width=1.0pt, color=green] (0,0.3-6.2) -- (5,0.3-6.2)node[anchor=west,color=purple]{$-\lambda^*/K(x)$};
	\draw [dotted, line width=1.0pt, color=green] (0,0.3-6.2) -- (0,5-6.2)node[anchor=south,color=black]{$p_0^*=0 \qquad$}node[anchor=north east,color=black]{\Large{B)} \qquad};
	
	\draw[blue,fill=blue] (0,5-6.2) circle (.25ex);
	\draw[blue,fill=blue] (5,3.7-6.2) circle (.25ex);
	%\draw[blue,fill=blue] (0,0.8+1.8*1.1*1.1-6.2) circle (.25ex);
	
	%\draw[blue,fill=blue] (0.15,5-6.2) circle (.25ex);
	%\draw[blue,fill=blue] (0.15,2.425-6.2) circle (.25ex);
	\draw[blue,fill=blue] (5,2.425-6.2) circle (.25ex);
	
	\draw[blue,fill=blue] (1.6,5-6.2) circle (.25ex);
	\draw[blue,fill=blue] (1.6,1.25-6.2) circle (.25ex);
	\draw[blue,fill=blue] (5,1.25-6.2) circle (.25ex);
	
	\draw[blue,fill=blue] (1.9,5-6.2) circle (.25ex);
	%\draw[blue,fill=blue] (1.9,1.952-6.2) circle (.25ex);
	\draw[blue,fill=blue] (5,0.3-6.2) circle (.25ex);
	
	%Plot 3
	%Horizontal axis
	\draw[->] (0,5-12.4) -- (5.5,5-12.4)node[anchor=west] {$p_0$};
	
	%Vertical axis
	\draw[->, color=purple] (5,0-12.4) -- (5,5.5-12.4); 
	
	%Plot	
	\draw [domain=0:1.9, line width=1.0pt, color=blue] plot ({\x},{-12.4+0.8+1.8*(\x^2-2*\x*1.1+1.1^2)});
	\draw [domain=1.9:4.2, line width=1.0pt, color=blue] plot ({\x},{-12.4+0.8+1.8*0.8^2+1.4*rad(atan(2.14*(\x-1.9)))});
	
	%Labels
	
	\draw [dashed, line width=1.0pt, color=black] (0,-12.4+0.8+1.8*1.1^2) -- (5,-12.4+0.8+1.8*1.1^2) node[anchor=west] {$w(0)$};
	\draw [dashed, line width=1.0pt, color=black] (0,-12.4+0.8+1.8*0.8^2+1.919) -- (5,-12.4+0.8+1.8*0.8^2+1.919) node[anchor=west] {$w(p_M)$};
	\draw [dashed, line width=1.0pt, color=black] (0,-12.4+0.8) -- (5,-12.4+0.8) node[anchor=west] {$\min_{p_0} w(p_0)$};
	
	\draw [dotted, line width=1.0pt, color=green] (4.2,4.4-12.4) -- (5,4.4-12.4)node[anchor=west,color=purple]{$-\lambda^*/K(x)$};
	\draw [dotted, line width=1.0pt, color=green] (4.2,4.4-12.4) -- (4.2,5-12.4)node[anchor=south,color=black]{$p_0^*=p_M$};
	
	\draw [dotted, line width=1.0pt, color=green] (1.55,1.16-12.4) -- (5,1.16-12.4)node[anchor=west,color=purple]{$-\lambda^*/K(x)$};
	\draw [dotted, line width=1.0pt, color=green] (1.55,1.16-12.4) -- (1.55,5.35-12.4)node[anchor=south,color=black]{$p_0^*\in\left\{0,w^{-1}\lrp{-\frac{\lambda^*}{K(x)}}\right\}$};
	
	\draw [dotted, line width=1.0pt, color=green] (2.7,3.41-12.4) -- (5,3.41-12.4)node[anchor=west,color=purple]{$-\lambda^*/K(x)$};
	\draw [dotted, line width=1.0pt, color=green] (2.7,3.41-12.4) -- (2.7,4.25-12.4)node[anchor=south,color=black]{\hspace{0.1cm}\small{$\quad p_0^*=w^{-1}\lrp{-\frac{\lambda^*}{K(x)}}$}};
	
	\draw [dotted, line width=1.0pt, color=green] (0,0.3-12.4) -- (5,0.3-12.4)node[anchor=west,color=purple]{$-\lambda^*/K(x)$};
	\draw [dotted, line width=1.0pt, color=green] (0,0.3-12.4) -- (0,5-12.4)node[anchor=south,color=black]{$p_0^*=0$}node[anchor=north east,color=black]{\Large{C)} \qquad};

	\draw[blue,fill=blue] (0,5-12.4) circle (.25ex);
	\draw[blue,fill=blue] (5,4.4-12.4) circle (.25ex);
	%\draw[blue,fill=blue] (0,0.8+1.8*1.1*1.1-12.4) circle (.25ex);
	
	\draw[blue,fill=blue] (1.55,5-12.4) circle (.25ex);
	\draw[blue,fill=blue] (1.55,1.16-12.4) circle (.25ex);
	\draw[blue,fill=blue] (5,1.16-12.4) circle (.25ex);
	
	\draw[blue,fill=blue] (2.7,5-12.4) circle (.25ex);
	\draw[blue,fill=blue] (2.7,3.41-12.4) circle (.25ex);
	\draw[blue,fill=blue] (5,3.41-12.4) circle (.25ex);
	
	\draw[blue,fill=blue] (4.2,5-12.4) circle (.25ex);
	%\draw[blue,fill=blue] (4.2,3.87-12.4) circle (.25ex);
	\draw[blue,fill=blue] (5,0.3-12.4) circle (.25ex); 
	
	\end{tikzpicture}
	\caption{Schematic representation of $w$, as a function of $p_0$ in case $T > T_0$. As $p_M$ increases (from top to bottom) the three diagrams, that we call A, B and C, show the three possible relative positions of $w(0)$, $w(p_M)$ and $\min_{p_0} w$.}\label{fig:TgeqT0}
\end{figure}

\subsection{The case $T\protect\leq T_{0}$}\label{sec:TleqT0}
First, we place ourselves in the case $T\leq T_{0}$. Let us introduce the following mappings defined on $\RR^+$
\begin{equation}\label{def:psi_smallT}
\Lambda\mapsto\psi_{x,T}(\Lambda):=\begin{cases}
0 & \text{if }-\Lambda\leq w_{T}(0),\\
p_{M}(x):= G^{-1}\left(\frac{M}{K(x)}\right) & \text{if }-\Lambda\geq w_{T}(p_{M}(x)),\\
w_{T}^{-1}(-\Lambda) & \text{if }-\Lambda\in\left(w_{T}(0),w_{T}(p_{M}(x))\right).
\end{cases}
\end{equation}
and
\begin{equation}\label{def:I_smallT}
\lambda\mapsto I(\lambda):=\int_{\Omega}K(x)G\lrp{\psi_{x,T}\lrp{\frac{\lambda}{K(x)}}}\,dx.
\end{equation}
In the following Theorem, we state the main result for the case $T\leq T_0$.
\begin{theorem}\label{theo:TleqT0} Assume $T\leq T_0$, $0<C<M|\Omega|$ and \eqref{H}. Then, there exists a unique $p_0^*\in\mathcal{P}_{0,C,M}$, that solves Problem \eqref{prob:p0}. It is given by $$p_0^*(x)=\psi_{x,T}\lrp{\frac{\lambda^{*}}{K(x)}} \ \text{ for any } \lambda^* \text{ such that } \ I(\lambda^{*})=C.$$
\end{theorem}
\begin{proof}
Fix any $x\in\Omega$ and denote 
\[
w(p_{0}):=w_T(p_{0}(x)),\qquad p_{M}=p_{M}(x):=G^{-1}\left(\frac{M}{K(x)}\right)>0.
\]
Let $\lambda^{*}\geq0$ be any value given by Lemma \ref{lem:w} and
assume $p_{0}^{*}=p_{0}^{*}(x)$ is any optimal control solving Problem \eqref{prob:p0}. The optimality conditions
given by Lemma \ref{lem:w} can be rewritten as:
\begin{itemize}
\item If $p_{0}^{*}=0,$ then $w(0)\geq-\frac{\lambda^{*}}{K(x)}$.
\item If $p_{0}^{*}=p_{M},$ then $w(p_{M})\leq-\frac{\lambda^{*}}{K(x)}$.
\item If $0<p_{0}^{*}<p_{M},$ then $w(p_{0}^{*})=-\frac{\lambda^{*}}{K(x)}$.
\end{itemize}
Now, fix $T\leq T_{0}$. To compute the value $p_{0}^{*}=p_{0}^{*}(x)$, first, we look at the function $p_{0}\to w(p_{0})$ for
$0\leq p_{0}\leq p_{M}$. Since, from Lemma \ref{lem:w_monotonicity},
$w$ is strictly increasing, we have $w(p_0)\in [w(0), w(p_M)]$, with $w(0) < w(p_M) < 0$. Now, depending on the value $\lambda^{*}$,
we distinguish three cases:
\begin{itemize}
\item If $-\frac{\lambda^{*}}{K(x)}\leq w(0)$, then necessarily $p_{0}^{*}=0$.
\item If $-\frac{\lambda^{*}}{K(x)}\geq w(p_{M})$, then necessarily $p_{0}^{*}=p_{M}$.
\item If $-\frac{\lambda^{*}}{K(x)}\in(w(0),w(p_{M}))$, then necessarily $p_{0}^{*}=w^{-1}\lrp{-\frac{\lambda^{*}}{K(x)}}$.
\end{itemize}
In other words, for any given $T\leq T_{0}$ and $x\in\Omega$ we have
\begin{equation}
p_{0}^{*}(x)=\psi_{x,T}\lrp{\frac{\lambda^{*}}{K(x)}}=\begin{cases}
0 & \text{if }-\frac{\lambda^{*}}{K(x)}\leq w_{T}(0),\\
p_{M}(x) & \text{if }-\frac{\lambda^{*}}{K(x)}\geq w_{T}(p_{M}(x)),\\
w_{T}^{-1}\lrp{-\frac{\lambda^{*}}{K(x)}} & \text{if }-\frac{\lambda^{*}}{K(x)}\in\left(w_{T}(0),w_{T}(p_{M}(x))\right).
\end{cases}\label{eq:p0*_smallT}
\end{equation}
As a consequence, if $u_{0}^{*}$ is an optimal control, $p_{0}^{*}$
must satisfy (\ref{eq:p0*_smallT}), meaning $p_{0}^{*}$ is uniquely
determined for a given $\lambda^{*}$. We claim that each value $\lambda^*$, given by Lemma \eqref{lem:w}, leads to the same function $p_0^*$, meaning $p_0^*$ is uniquely determined. Consider $I(\lambda)$ as defined in \eqref{def:I_smallT}. If $p_{0}^{*}$ is optimal, then necessarily $I\lrp{\lambda^{*}}=C$,
see Lemma \ref{lem:constr_sat}. Note that the function $\psi_{x,T}$
is clearly continuous and nonincreasing, thus so is $I$. Moreover,
\[
I\lrp{\lambda}=\int_{\Omega}K(x)G(0)dx=0,\qquad\text{if }\lambda\geq\lambda_{max}:=-w_{T}(0)\min_{x}K(x),
\]
\[
I\lrp{\lambda}=\int_{\Omega}K(x)G(p_{M}(x))dx=M|\Omega|,\qquad\text{if }\lambda\leq\lambda_{min}:=-\max_{x}K(x)w_{T}(p_{M}(x)).
\]
Since we assumed $0<C<M|\Omega|$, we deduce that there exist $\lambda_{min}<\lambda_{1}^{*}\leq\lambda_{2}^{*}<\lambda_{max}$
such that 
\[
I\lrp{\lambda}=C\qquad\forall\lambda\in[\lambda_{1}^{*},\lambda_{2}^{*}].
\]
Therefore, $\lambda^{*}\in[\lambda_{1}^{*},\lambda_{2}^{*}]$.
While $\lambda^{*}$ is not uniquely determined, we claim that $\psi_{x,T}$
is constant on $[\lambda_{1}^{*}/K(x),\lambda_{2}^{*}/K(x)]$ for a.e. $x\in\Omega$.
Assume by contradiction that there exists a set $S\subset\Omega$
with positive measure such that $\psi_{x,T}$ is nonconstant on $[\lambda_{1}^{*}/K(x),\lambda_{2}^{*}/K(x)]$
for all $x\in S$. This implies, since $\psi_{x,T}$ is nonincreasing,
\[
\psi_{x,T}\lrp{\frac{\lambda_1^*}{K(x)}}>\psi_{x,T}\lrp{\frac{\lambda_2^*}{K(x)}},\qquad\forall x\in S.
\]
On the other hand, we also have
\[
\psi_{x,T}\lrp{\frac{\lambda_1^*}{K(x)}}\geq\psi_{x,T}\lrp{\frac{\lambda_2^*}{K(x)}},\qquad\forall x\in\Omega.
\]
As a result, since $G$ is increasing, we deduce that 
\begin{align*}
I\lrp{\lambda_1^*}-I\lrp{\lambda_2^*} & =\int_{\Omega\backslash S}K(x)\left(G\lrp{\psi_{x,T}\lrp{\frac{\lambda_1^*}{K(x)}}}-G\lrp{\psi_{x,T}\lrp{\frac{\lambda_2^*}{K(x)}}}\right)\,dx \\
& \quad  +\int_{S}K(x)\left(G\lrp{\psi_{x,T}\lrp{\frac{\lambda_1^*}{K(x)}}}-G\lrp{\psi_{x,T}\lrp{\frac{\lambda_2^*}{K(x)}}}\right)\,dx\\
 & \geq  \int_{S}K(x)\left(G\lrp{\psi_{x,T}\lrp{\frac{\lambda_1^*}{K(x)}}}-G\lrp{\psi_{x,T}\lrp{\frac{\lambda_2^*}{K(x)}}}\right)\,dx \\
 & >0
\end{align*}
where the last inequality follows from the fact that $|S|>0$ and
$K(x)>\min_{\Omega}K>0$. This contradicts the fact that $I\lrp{\lambda_1^*}=I\lrp{\lambda_2^*}=C$.
Therefore, $\psi_{x,T}$ is constant on $[\lambda_{1}^{*}/K(x),\lambda_{2}^{*}/K(x)]$
for a.e. $x\in\Omega$. As a result, $p_{0}^{*}(x)=\psi_{x,T}\lrp{\frac{\lambda^*}{K(x)}}$
is uniquely determined, for any value $\lambda^{*}\in[\lambda_{1}^{*},\lambda_{2}^{*}]$. Moreover, by the definition of $\psi_{x,T}(\cdot)$, we have that $0\leq p_0^*(x)\leq G^{-1}\lrp{\frac{M}{K(x)}}$ a.e., and we already proved that $\int_{\Omega}K(x)G(p_0^*(x))dx=C$, thus $p_0^*\in\mathcal{P}_{0,C,M}$. Since we established that the constructed $p_0^*$ is the only possible solution to Problem \eqref{prob:p0} and $p_0^*\in\mathcal{P}_{0,C,M}$, we conclude not only the existence but also the uniqueness of the solution up to a rearrangement in the parts of the domain $\Omega$ where $K(x)$ is constant.
\end{proof}

\begin{remark}\normalfont{
  Notice that Theorem \ref{theo:TleqT0} implies that releases should be more important where the carrying capacity is high. Since $\lambda^*$ is fixed, the argument of  $\psi_{x,T}(\cdot)$, $\lambda^*/K(x)$, is smaller where $K(x)$ is higher. In case $-\frac{\lambda^*}{K(x)}\not\in\left(w_{T}(0),w_{T}(p_{M}(x))\right)$, we have either $u_0(x)=0$ or $u_0(x)=M$. On the other hand, in case  $-\frac{\lambda^*}{K(x)}\in\left(w_{T}(0),w_{T}(p_{M}(x))\right)$,  $\psi_{x,T}\lrp{\frac{\lambda^*}{K(x)}}=w^{-1}_{T}\lrp{-\frac{\lambda^*}{K(x)}}$. And since $w_{T}(\cdot)$ is monotonically increasing (see Figure \ref{fig:TleqT0}), a bigger $K(x)$ implies a bigger argument (because of the minus sign), which implies a bigger $p_0^*(x)$. Since $u_0^*(x)=K(x)G(p_0^*(x))$, and $G$ is also monotonically increasing, it follows that $u^*_0(x)$ must be non-decreasing when $K(x)$ increases in general, and strictly increasing with $K(x)$ whenever $u_0^*(x)\not\in\left\{0,M\right\}$.}
\end{remark}

\subsection{The case $T > T_{0}$}\label{sec:TgeqT0}

We study now the case $T > T_{0}$. In this case, we have from Lemma \ref{lem:w_monotonicity} that the function $w_T$ is decreasing on $(0,p_0^T)$ and increasing on $(p_0^T,1)$; hence it is not injective. Then we consider its restriction on $(p_0^T,1)$ and consider its inverse $w_T^{-1}:(w_T(p_0^T),0) \to (p_0^T,1)$, which is an increasing function.

In order to state the results in this case it will be useful to introduce some tools and notation. Let us introduce the following mappings
\begin{equation}\label{def:psi0_bigT}
\Lambda\mapsto\psi^0_{x,T}(\Lambda):=\begin{cases}
0 & \text{if }-\Lambda \leq w_{T}(0),\\
p_{M}(x) & \text{if }-\Lambda > \max\lrp{w_{T}(0),w_{T}(p_{M}(x))},\\
w_{T}^{-1}(-\Lambda) & \text{if }-\Lambda\in\left(w_{T}(0),w_{T}(p_{M}(x))\right],
\end{cases}
\end{equation}
the third case only being defined if $w_{T}(0)<w_{T}(p_{M}(x))$, and
\begin{equation}\label{def:psi1_bigT}
\Lambda\mapsto\psi^1_{x,T}(\Lambda):=\begin{cases}
0 & \text{if }-\Lambda < \min\limits_{p_0\in(0,p_M(x))} w_{T}(p_0),\\
p_{M}(x) & \text{if }-\Lambda\geq w_{T}(p_{M}(x)),\\
w_{T}^{-1}(-\Lambda) & \text{if }-\Lambda\in\left[\min\limits_{p_0\in(0,p_M(x))} w_{T}(p_0),w_{T}(p_{M}(x))\right),
\end{cases}
\end{equation}
the third case only being defined if $\min\limits_{p_0\in(0,p_M(x))} w_{T}(p_0)<w_{T}(p_{M}(x))$. It is important to remark that $w^{-1}\lrp{-\frac{\lambda^*}{K(x)}}$ might not be uniquely defined, since the function is not injective on its entire domain. Whenever there is an ambiguity it will be understood that the value of $w^{-1}\lrp{-\frac{\lambda^*}{K(x)}}$ we refer to, is the one on the increasing branch of $w(p_0)$ (the one satisfying $\frac{\partial w}{\partial p_0}(p_0)\geq 0$), since it is the only one satisfying the second order optimality conditions.

For a given value of $\lambda\geq 0$, let us introduce the set:
\begin{equation}\label{def:tilde_Omega}
	\tilde{\Omega}_{\lambda}:=\left\{x\in \Omega \ | \   \psi^0_{x,T}\lrp{\frac{\lambda}{K(x)}}\neq\psi^1_{x,T}\lrp{\frac{\lambda}{K(x)}}  \right\}.
\end{equation}
By definition, for all $x\in\Omega\setminus\tilde{\Omega}_{\lambda}$, $\psi^0_{x,T}\lrp{\frac{\lambda}{K(x)}}=\psi^1_{x,T}\lrp{\frac{\lambda}{K(x)}}$. In order to underline this, for $x\in\Omega\setminus\tilde{\Omega}_{\lambda}$ we will denote $\psi^{\sbullet}_{x,T}\lrp{\frac{\lambda}{K(x)}}:=\psi^0_{x,T}\lrp{\frac{\lambda}{K(x)}}=\psi^1_{x,T}\lrp{\frac{\lambda}{K(x)}}$. Note also that in case $\psi^0_{x,T}\lrp{\frac{\lambda}{K(x)}}\neq\psi^1_{x,T}\lrp{\frac{\lambda}{K(x)}}$, then $\psi^0_{x,T}\lrp{\frac{\lambda}{K(x)}}=0$. Therefore, $\psi^0_{x,T}\lrp{\frac{\lambda}{K(x)}}=0$ for all $x\in\tilde{\Omega}_{\lambda}$. To see this, notice that $\max\lrp{w_{T}(0),w_{T}(p_{M}(x))}\geq w_{T}(p_{M}(x))$, therefore if $\psi^0_{x,T}\lrp{\Lambda}=p_M(x)$, then $\psi^1_{x,T}(\Lambda)=p_M(x)$ too. Also, in case $w_{T}(0)<w_{T}(p_{M}(x))$, then $\left(w_{T}(0),w_{T}(p_{M}(x))\right)\subset \left[\min_{p_0\in(0,p_M(x))} w_{T}(p_0),w_{T}(p_{M}(x))\right)$, therefore if $\psi^0_{x,T}\lrp{\Lambda}=w^{-1}_T(\Lambda)$, then $\psi^1_{x,T}\lrp{\Lambda}=w^{-1}_T(\Lambda)$ too. Finally, if $\Lambda=w_T(p_M(x))$, then $\psi^0_{x,T}\lrp{\Lambda}=w^{-1}_T(w_T(p_M(x)))=p_M(x)=\psi^1_{x,T}\lrp{\Lambda}$.

In the same spirit as in Theorem \ref{theo:TleqT0}, the idea behind Theorem \ref{theo:TgeqT0} is to write the solution in the form $p_0^*(x)=\psi^{\sbullet}_{x,T}\lrp{\frac{\lambda}{K(x)}}$ for certain values of $\lambda$. Therefore, solutions in $\tilde{\Omega}_{\lambda}$ will be hard to characterize in general. In order to study solutions in this set we introduce a secondary problem, the solution of which will allow us to determine the solutions of Problem \eqref{prob:p0} in $\tilde{\Omega}_{\lambda}$. Assuming $|\tilde{\Omega}_{\lambda}|>0$, we introduce the quantity $$\tilde{C}_{\lambda}:=C-\int_{\Omega\setminus\tilde{\Omega}_{\lambda}}K(x)G\lrp{\psi^{\sbullet}_{x,T}\lrp{\frac{\lambda}{K(x)}}}\,dx.$$ 
Assuming $\tilde{C}_{\lambda}>0$ we consider the following problem: 
\paragraph{Secondary problem.}
\begin{equation}\label{prob:shape}\tag{$\mathcal{P}_{\tilde{\Omega}_{\lambda}}$}
\resizebox{0.91\textwidth}{!}{\boxed{
		\inf_{\chi\in \mathcal{X}}\int_{\tilde{\Omega}_{\lambda}}K(x)^2(1-p(T,x))^2\chi(x)+K(x)^2(1-\chi(x)) dx \ , \ \int_{\tilde{\Omega}_{\lambda}}K(x)G\lrp{\psi^1_{x,T}\lrp{\frac{\lambda^*}{K(x)}}}\chi(x) dx \leq \tilde{C}_{\lambda},
}}
\end{equation}
where the new control variable is $\chi_{\lambda}\in\mathcal{X}$ and $$\mathcal{X}:=\left\{\chi_{\lambda}\in L^{\infty}(\tilde{\Omega}_{\lambda}) \ | \  0\leq\chi_{\lambda}\leq 1 \right\}.$$ Here $p(T,x)$ is assumed to have $\psi^1_{x,T}\lrp{\frac{\lambda}{K(x)}}$ as initial condition.

Finally, let us also define
\begin{equation}\label{def:I0_I1}
I^0(\lambda):=\int_{\Omega}K(x)G\lrp{\psi^0_{x,T}\lrp{\frac{\lambda}{K(x)}}}dx \quad \mbox{and} \quad I^1(\lambda):=\int_{\Omega}K(x)G\lrp{\psi^1_{x,T}\lrp{\frac{\lambda}{K(x)}}}dx.
\end{equation}
and
\begin{equation}\label{def:lambda0}
\lambda_0:=\min\left\{\lambda \quad \text{such that} \quad I^0(\lambda)=C\right\},
\end{equation}
\begin{equation}\label{def:lambda1} 
\lambda_1:=\max \left\{\lambda \quad \text{such that} \quad I^1(\lambda)=C\right\}.
\end{equation}
Note that as long as $0<C<M|\Omega|$ these two quantities will always be well defined.

\begin{theorem}\label{theo:TgeqT0}
	Assume $T > T_0$, $0<C<M|\Omega|$ and \eqref{H}. Then, any solution of Problem \eqref{prob:p0}, $p_0^*\in\mathcal{P}_{0,C,M}$, has to satisfy: 
	$$p_0^*(x)=\psi_{x,T}^{\sbullet}\lrp{\frac{\lambda^*}{K(x)}} \text{ for all }  x\in\Omega\setminus\tilde{\Omega}_{\lambda^*}, \mbox{ and } p_0^*(x)=\begin{cases} 
	\psi_{x,T}^{1}\lrp{\frac{\lambda^*}{K(x)}}, & \text{for } x\in\tilde{\Omega}_{\lambda^*} \mbox{ s.t. } \chi_{\lambda^*}^*(x)=1,\\ 
	\psi_{x,T}^0\lrp{\frac{\lambda^*}{K(x)}}, & \text{for } x\in\tilde{\Omega}_{\lambda^*} \mbox{ s.t. } \chi_{\lambda^*}^*(x)=0,\\
	\end{cases}$$
	where $\psi_{x,T}^0\lrp{\frac{\lambda^*}{K(x)}}=0$ for all $x\in\tOm$, $\lambda^*\in[\lambda_0,\lambda_1]$ and $\chi^*_{\lambda^*}(x)$ is the solution to Problem \eqref{prob:shape} with $\lambda=\lambda^*$.
	
	Furthermore, if $|\tilde{\Omega}_{\lambda_0}|=0$, then $\lambda^*=\lambda_0$ and $p_0^*(x)=\psi_{x,T}^{\sbullet}\lrp{\frac{\lambda_0}{K(x)}}$  for all $x\in\Omega$.
\end{theorem}

\begin{proof}

In this proof we use the same notation as in the proof of Theorem \ref{theo:TleqT0}. Let us fix any $x\in\Omega$ and let $\lambda^{*}\geq0$ be any value given by Lemma \ref{lem:w}. Let  $p_0^*\in\mathcal{P}_{0,C,M}$ be the solution of Problem \eqref{prob:p0}.

Using the first and second order optimality conditions (see Lemma \eqref{lem:w}) we know that there exists a $\lambda^*\geq 0$ such that the optimal control $p_0^*$ must satisfy:
\begin{itemize}
	\item If $p_{0}^{*}=0$ then $w(0)\geq-\frac{\lambda^*}{K(x)}$.
	\item If $p_{0}^{*}=p_{M}$ then $w(p_{M})\leq-\frac{\lambda^*}{K(x)}$.
	\item If $0<p_{0}^{*}<p_{M}$ then $w(p_{0}^{*})=-\frac{\lambda^*}{K(x)}$ and $\frac{\partial w}{\partial p_0}(p_0^*)\geq 0$.
\end{itemize}

We fix $T>T_0$ and exploit these optimality conditions. Under hypothesis \eqref{H}, by Lemma \ref{lem:w_monotonicity}, $w(p_0)$ is unimodal, that is, strictly decreasing until a certain $p_0^T\in(0,\bar{\theta})$ and then strictly increasing. Depending on the relative position of $w(0)$,$w(p_M)$ and $\min w(p_0)$ we can have three different behaviors of the optimal control. We detail here as an example the case $w(0)\geq w(p_M)>\min w(p_0)$. Cases $w(0)>w(p_M)\geq \min w(p_0)$ and $w(p_M)>w(0)>\min w(p_0)$ can be deduced from this analysis and Figure \ref{fig:TgeqT0}.

Assume $p_M$ is such that $w(0)\geq w(p_M)>\min w(p_0)$ (case B in Figure \ref{fig:TgeqT0}), then we distinguish four cases:
\begin{itemize}
	\item If $-\frac{\lambda^*}{K(x)}\geq w(0)$, then necessarily $p_{0}^{*}=p_M$.
	\item If $w(0) \geq -\frac{\lambda^*}{K(x)}\geq w(p_{M})$, then necessarily $p_{0}^{*}\in\lrk{0,p_{M}}$.
	\item If $w(p_M)\geq -\frac{\lambda^*}{K(x)}\geq \min w(p_0)$, then necessarily $p_{0}^{*}\in\lrk{0,w^{-1}\lrp{-\frac{\lambda^*}{K(x)}}}$.
	\item If $-\frac{\lambda^*}{K(x)} \leq \min w(p_0)$, then necessarily $p_{0}^{*}=0$.
\end{itemize}

Two things remain to be studied: The values that $\lambda^*$ can take, and, in case $p_0^*$ is not uniquely determined, how to choose between the two options. In order to do this, let us consider mappings $\psi_{x,T}^0$ and $\psi_{x,T}^1$ as defined in \eqref{def:psi0_bigT} and \eqref{def:psi1_bigT} respectively. Note that these two mappings are always well defined and they give, respectively, the minimum and maximum values $p_0^*$ can take when two values of $p_0$ satisfy the optimality conditions. For instance, if for a given value of $\lambda^*$, $p_0^*\in\lrk{0,p_M}$, then $\psi^0_{x,T}\lrp{\frac{\lambda^*}{K(x)}}=0$ and $\psi^1_{x,T}\lrp{\frac{\lambda^*}{K(x)}}=p_M$. Remark also that whenever $\psi^0_{x,T}\lrp{\frac{\lambda}{K(x)}}\neq\psi^1_{x,T}\lrp{\frac{\lambda}{K(x)}}$, $\psi^0_{x,T}\lrp{\frac{\lambda}{K(x)}}=0$.

Mappings $\psi^0_{x,T}\lrp{\frac{\lambda}{K(x)}},\psi^1_{x,T}\lrp{\frac{\lambda}{K(x)}}$ are non-increasing with respect to $\lambda$. Although in case $T>T_0$, $\psi^0_{x,T}$ and $\psi^1_{x,T}$ are only continuous in case C (see Figure \ref{fig:TgeqT0}), it still holds that for all $\lambda\geq 0$ we have that $I^0(\lambda)\leq I^1(\lambda)$, with $I^0(\lambda), I^1(\lambda)$, as defined in \eqref{def:I0_I1}. Furthermore, if $\lambda^*$ is any value given by Lemma \ref{lem:w} we have that
\begin{equation} \label{I0_KG_I1}
	I^0(\lambda^*)\leq \int_{\Omega}K(x)G(p_0^*(x))dx \leq I^1(\lambda^*).
\end{equation}
%By the assumption $0<C<M|\Omega|$, we know that there must exist $\lambda_{min}$ and $\lambda_{max}$ such that
%\begin{equation}\label{def:lambdamin}
%I^0(\lambda)=\int_{\Omega}K(x)G(p_{M}(x))dx=M|\Omega|,\qquad\text{if }\lambda\leq\lambda_{min}:=-\max_{x} \max_{p_0} w_{x,T}(p_{M}(x)).
%\end{equation}
%\begin{equation}\label{def:lambdamax}
%I^1(\lambda)=\int_{\Omega}K(x)G(0)dx=0,\qquad\text{if }\lambda\geq\lambda_{max}:=-\min_{x} \min_{p_0} w_{x,T}(0),
%\end{equation}
%Due tho this, and since by Lemma \ref{lem:constr_sat}, necessarily $\int_{\Omega}K(x)G(p_0^*(x))dx=C$, we deduce that $\lambda^*\in[\lambda_{min},\lambda_{max}]$.
Using Lemma \ref{lem:constr_sat}, \eqref{I0_KG_I1} means that $I^0(\lambda^*)\leq C \leq I^1(\lambda^*)$. It follows that $\lambda^*\in[\lambda_{0},\lambda_{1}]$. 

Fixing now a $\lambda^*\in[\lambda_{0},\lambda_{1}]$, let us consider the set $\tilde{\Omega}_{\lambda^*}$ as defined in \eqref{def:tilde_Omega}. Note that if $|\tilde{\Omega}_{\lambda^*}|=0$, we can conclude that $p_0^*(x)=\psi^{\sbullet}_{x,T}\lrp{\frac{\lambda^*}{K(x)}}$ using the same arguments as in the proof of Theorem \ref{theo:TleqT0}. In particular, if $|\tilde{\Omega}_{\lambda_0}|=0$, since $p_0^*(x)=\psi^{\sbullet}_{x,T}\lrp{\frac{\lambda_0}{K(x)}}$, $\psi^{\sbullet}_{x,T}(\lambda)$ is non-increasing w.r.t $\lambda$ and $\int_{\Omega}K(x)^2\lrp{1-p(T,x)}^2 \,dx$ is descreasing w.r.t to the initial data of $p(T,x)$, we can conclude that $\lambda^*=\lambda_0$, proving the last statement of the Theorem. 

For the rest of the proof we assume that $|\tilde{\Omega}_{\lambda^*}|>0$. Note also that the optimal control $p_0^*$ must be such that $0 < \tilde{C}_{\lambda^*}\leq C$. We consider Problem \eqref{prob:shape} with $\lambda=\lambda^*$. Observing that the criterion to optimize is affine with respect to $\chi_{\lambda^*}$ and that its differential at $\chi_{\lambda^*}^*$ is the linear mapping
$$
L^\infty(\Omega) \owns h\mapsto  \int_{\tOm} h \ K(x)\left(K(x)p(T,x)(p(T,x)-2) + \tilde{\lambda}G\lrp{\psi^1_{x,T}\lrp{\frac{\lambda^*}{K(x)}}}\right) \ dx,
$$
leads to introduce the switch function $\Phi$ for this problem, namely 
$$
\Phi(x):=K(x) \ \frac{p(T,x)(2-p(T,x))}{G\lrp{\psi^1_{x,T}\lrp{\frac{\lambda^*}{K(x)}}}}.
$$

We infer from the so-called bathtub principle (see e.g. Section 1.14 of \cite{Bathtub}) the existence of a unique real number $\tilde{\lambda}^*$ such that
$$
\{\tilde{\lambda}^*> \Phi\}\subset \left\{\chi_{\lambda^*}^*=0\right\}\subset \{\tilde{\lambda}^*\geq \Phi\}, \qquad \{\tilde{\lambda}^*< \Phi\}\subset \left\{\chi_{\lambda^*}^*=1\right\}\subset \{\tilde{\lambda}^*\leq \Phi\}
$$
and furthermore, $\left\{0<\chi_{\lambda^*}^*<1\right\}\subset \{\tilde{\lambda}^*= \Phi\}$. Note that such inclusions must be understood up to a zero Lebesgue measure set.

Let us denote $D:=\left\{x\in\tOm \ | \ 0<\chi_{\lambda^*}^*(x)<1\right\}$. In case $|D|=0$, the optimality conditions become necessary and sufficient. The solution can be written as $$\chi^*_{\lambda^*}(x)=\begin{cases} 1 &, \text{if } \tilde{\lambda}^* < \Phi,\\
0 &, \text{if } \tilde{\lambda}^* > \Phi.\\
\end{cases}$$

In case $|D|>0$, since the problem is linear in $\chi_{\lambda^*}$ we know that there exists a bang-bang solution. That is, a solution that only takes the values $\chi_{\lambda^*}^*=0$ and $\chi_{\lambda^*}^*=1$. This means that despite it may exist a solution with $0<\chi_{\lambda^*}^*(x)<1$ for $x\in D$, we can always construct a bang-bang alternative that performs just as good.
Assuming $\tilde{\lambda}^*=\Phi$ in $D$ ($\Phi$ is constant in $D$ by definition), we introduce $$\chi^{\alpha}_{\lambda^*}(x)=\begin{cases} 1, & \text{if } \tilde{\lambda}^* < \Phi \text{ or } x \in D^{\alpha},\\
0, & \text{if } \tilde{\lambda}^* > \Phi \text{ or } x \in D \setminus D^{\alpha}.
\end{cases}$$
where $D^{\alpha}$ is any subset of $D$ such that $|D^{\alpha}|/|D|=\alpha$, and $\alpha\in[0,1]$. To compute the value of $\alpha$ we use one more time Lemma \ref{lem:constr_sat}, concluding that, in this case, $\chi_{\lambda^*}^*(x)=\chi_{\lambda^*}^{\alpha}(x)$ for $\alpha$ such that $\int_{\tilde{\Omega}_{\lambda^*}}K(x)G\lrp{\psi^1_{x,T}\lrp{\frac{\lambda^*}{K(x)}}}\chi^{\alpha}_{\lambda^*}(x) dx = \tilde{C}_{\lambda^*}$.
Note how the solution to this secondary problem, \eqref{prob:shape}, sheds light on the primary problem, \eqref{prob:p0}, by allowing us to write the solution on $\tOm$ as 
$$
p_0^*(x)=\begin{cases} 
\psi^1_{x,T}(\lambda^*), & \mbox{ on } \chi_{\lambda^*}^*(x)=1,\\ 
\psi^0_{x,T}(\lambda^*), & \text{ on } \chi_{\lambda^*}^*(x)=0,\\
\end{cases}
$$
concluding the proof.
\end{proof}
\begin{remark}\normalfont{
    We stress that in Theorem \ref{theo:TgeqT0} we do not establish the existence of solutions for Problem \eqref{prob:p0} in all generality. Comments on the existence of solutions for this case ($T>T_0$) are made in Appendix \ref{app:Existence}. Additionally, the existence of solutions in case $K$ is a piece-wise constant function is proven.
	}
\end{remark}

The last result concerns the particular case where $K$ is a constant.
\begin{corollary}\label{coro:solKconstant}
   
    Assume $0<C<M|\Omega|$ and $K(x)=K$ constant in all $\Omega$. Then there exists a $p_0^*\in\mathcal{P}_{0,C,M}$ that solves Problem \eqref{prob:p0}.

    \begin{itemize}
        \item If $T\leq T_0$, it is given by $$p_0^*(x)=G^{-1}\lrp{\frac{C}{K|\Omega|}} \text{ for all }  x\in\Omega.$$  
        \item If $T>T_0$,
        \begin{itemize}
        	\item If $w(0)<w\lrp{G^{-1}\lrp{\frac{C}{K|\Omega|}}}$, then $$p_0^*(x)=G^{-1}\lrp{\frac{C}{K|\Omega|}} \text{ for all }  x\in\Omega$$
        	\item If $w(0) \geq w\lrp{G^{-1}\lrp{\frac{C}{K|\Omega|}}}$, then there exists at least one $\lambda^*\in[\lambda_0,\lambda_1]$ such that $p_0^*(x)$ can be written as $$p_0^*(x)=\begin{cases} 
        	\psi_{x,T}^1\lrp{\frac{\lambda^*}{K}}, & \text{for } x\in D,\\ 
        	0, & \text{for } x\in \Omega\setminus D,\\
        	\end{cases}$$
        	where $D$ can be any subdomain of $\Omega$ with size $|D|=\frac{C}{KG\lrp{\psi_{x,T}^1\lrp{\frac{\lambda^*}{K}}}}$.
        \end{itemize} 
       
    \end{itemize}
    
\end{corollary}

\begin{proof}

The existence of a unique solution written as $p_0^*(x)=\psi_{x,T}\lrp{\frac{\lambda^*}{K(x)}}$ for the case $T\leq T_0$ can be easily adapted from the proof of Theorem \ref{theo:TleqT0}. Since the constraint must be saturated, we have that $$\int_{\Omega}K(x)G\lrp{\psi_{x,T}\lrp{\frac{\lambda^*}{K(x)}}}\,dx=K\int_{\Omega}G\lrp{\psi_{x,T}\lrp{\frac{\lambda^*}{K}}}\,dx=C,$$ but for $K$ constant, $w$ is constant w.r.t. $x$, and thus, so is $\psi_{x,T}\lrp{\frac{\lambda^*}{K}}$. Therefore, $$K\int_{\Omega}G\lrp{\psi_{x,T}\lrp{\frac{\lambda^*}{K}}}\,dx=KG\lrp{\psi_{x,T}\lrp{\frac{\lambda^*}{K}}}|\Omega|=C,$$
which yields $$p_0^*(x)=\psi_{x,T}\lrp{\frac{\lambda^*}{K}}=G^{-1}\lrp{\frac{C}{K|\Omega|}}.$$

The case $T>T_0$ is greatly simplified in this setting . Since $\psi_{x,T}\lrp{\frac{\lambda^*}{K}}$ is constant w.r.t. $x$, either $|\tilde{\Omega}_{\lambda}|=0$, or $\tilde{\Omega}_{\lambda}=\Omega$. Like in the general case, if $|\tilde{\Omega}_{\lambda_0}|=0$ we have that $p_0^*(x)=\psi^{\sbullet}_{x,T}\lrp{\frac{\lambda_0}{K}}=G^{-1}\lrp{\frac{C}{K|\Omega|}}$, and by the monotonicity of $\int_{\Omega}(1-p(T,x))^2\,dx$ w.r.t to the initial condition of $p(T,x)$ we can conclude. On the other hand, in this case, we can put the condition $|\tilde{\Omega}_{\lambda_0}|=0$ in simpler terms. If $|\tilde{\Omega}_{\lambda_0}|=0$, then $\psi_{x,T}^0\lrp{\frac{\lambda_0}{K}}=\psi_{x,T}^1\lrp{\frac{\lambda_0}{K}}$. Looking at the two functions, this happens if and only if $w(0)<-\frac{\lambda_0}{K}$ and we have that $$w(0)<-\frac{\lambda_0}{K}=w(p_0^*(x))=w\lrp{\psi^{\sbullet}_{x,T}\lrp{\frac{\lambda_0}{K}}}=w\lrp{G^{-1}\lrp{\frac{C}{K|\Omega|}}}.$$ Therefore, $$|\tilde{\Omega}_{\lambda_0}|=0 \text{ if and only if } w(0)< w\lrp{G^{-1}\lrp{\frac{C}{K|\Omega|}}}.$$

In case $ w(0)\geq w\lrp{G^{-1}\lrp{\frac{C}{K|\Omega|}}}$, we have that $|\tilde{\Omega}_{\lambda_0}|>0$ and thus, $|\tOm|>0$. Furthermore, since $K$ is constant, $\tOm=\Omega$.

Fixing a value for $\lambda^*$, $p_0^*(x)$ can only take two values in $\Omega$, $p_0^*(x)=0$ or $p_0^*(x)=\psi_{x,T}^1(\frac{\lambda^*}{K})$. Therefore, from Lemma \ref{lem:constr_sat} we can directly deduce the size of the domain where $p_0^*(x)=\psi_{x,T}^1\lrp{\frac{\lambda^*}{K}}$, that we denote $D$. We have $$\int_{\Omega}KG(p_0^*(x))\,dx=\int_{D}KG\lrp{\psi_{x,T}^1\lrp{\frac{\lambda^*}{K}}}\,dx=|D|KG\lrp{\psi_{x,T}^1\lrp{\frac{\lambda^*}{K}}}=C,$$
thus, $$|D|=\frac{C}{KG(\psi_{x,T}^1\lrp{\frac{\lambda^*}{K}})}.$$ Therefore, the solution can be written as $\displaystyle p^*_0(x)=\psi_{x,T}^1\lrp{\frac{\lambda^*}{K}} \text{ a.e. on } D,$ and $p_0^*(x)=0$ elsewhere, with $|D|=\frac{C}{KG\lrp{\psi_{x,T}^1\lrp{\frac{\lambda^*}{K}}}}$.

\end{proof}

\section{Numerical Implementation of results}\label{sec:numeric}

Thanks to Theorems \ref{theo:TleqT0} and \ref{theo:TgeqT0} we can implement an algorithm for computing solutions to problems \eqref{prob:u0} and \eqref{prob:p0}. In the case $T\leq T_0$, the computations will be a direct application of the results presented in Theorem \ref{theo:TleqT0}, where solutions are unique up to a rearrangement. In the case $T>T_0$, Theorem \ref{theo:TgeqT0} allows us to significantly simplify the problem by recasting it as a one-dimensional one when solutions cannot be found directly. Namely, 
\begin{equation}\label{prob:lambda}\tag{$\mathcal{Q}$}
\boxed{
	\inf_{\lambda\in[\lambda_0,\lambda_1]} \int_\Omega K(x)^2 (1-p(T,x))^2\,dx
}
\end{equation}
where we assume that the initial condition for $p(T,x)$ is given by the optimal release strategy, $p_0^*(x)$, given by theorem \ref{theo:TgeqT0}, assuming that $\lambda=\lambda^*$. Following the lines of \cite{DHP}, we present the simulations in the context of mosquito population replacement using \textit{Wolbachia}, although as discussed in Section \ref{sec:Intro}, the case with diffusion is more relevant in this setting. The parameters used for the simulations are presented in Table \ref{tab:Parameters}.

\begin{table}[htb!]
	\centering
	\begin{tabular}{|c|c|c|c|}
		\hline
		Category & Parameter & Name & Value \\
		\hline 
		\multirow{5}{*}{Optimization} &  $T$ & Final time & $\left\{1,25\right\}$ \\
		\cline{2-4}
		&  $M$ & Maximal instantaneous release rate & $250$ \\ 
		\cline{2-4}
		&  $C$ & Amount of mosquitoes avaliable  & $\left\{30,200\right\}$  \\
		\cline{2-4}
		&  $|\Omega|$ & Domain size  & $1$  \\
		\cline{2-4}
		& $K_0$ & Average carrying capacity & 100 \\
		\hline
		\multirow{5}{*}{Biology}&  $b_1^0$ & Normalized wild birth rate & 1 \\
		\cline{2-4}
		& $b_2^0$ & Normalized infected birth rate & 0.9 \\
		\cline{2-4}
		& $d_1$ & Wild death rate & 0.27 \\ 
		\cline{2-4}
		& $d_2$ & \textit{Wolbachia} infected death rate & 0.3 \\
		\cline{2-4}
		& $s_h$ & Cytoplasmic incompatibility level & 0.9 \\
		\hline
	\end{tabular}
	\caption{Values of the parameters used in simulations.  The values for the biological parameters have been taken from \cite{APSV}.}
	\label{tab:Parameters}
\end{table}

\subsection{1D simulations}\label{sec:1D_Sims}

The simplest setting in which we can study the problem is considering only one spatial dimension. We present two examples of the solutions obtained exploiting the results proven in this paper in a 1D setting. We consider the following functions representing the carrying capacity of the environment
\begin{equation}\label{K_simulations}
	K_S(x)=K_0\lrp{1-\frac{1}{2}\cos\lrp{\frac{2\pi x}{|\Omega|}}}, \quad K_P(x)=\frac{3K_0}{2}\mathbbm{1}_{\left[0,|\Omega|/2\right]}+\frac{K_0}{2}\mathbbm{1}_{\left(|\Omega|/2,|\Omega|\right]}
\end{equation}

With the parameters considered, the two functions have the same average carrying capacity, $K_0$. That is, $\int_{\Omega}K_S(x)\,dx=\int_{\Omega}K_P(x)\,dx=K_0|\Omega|$, which is also the value we would obtain in case of an homogeneous carrying capacity equal to $K_0$ in all the domain. The domain considered is $\Omega=[0,|\Omega|]$. We chose these two functions in order to have a piece-wise constant function and a function that is non-constant in any positive measure interval for comparison.

For the parameters in Table \ref{tab:Parameters}, $A(\cdot)$, as defined in \eqref{def:A}, satisfies hypothesis \eqref{H}. The time when the function stops being increasing and starts being unimodal (decreasing, then increasing) is $T_0\approx 3.51$, computed using formula \eqref{eq:T0}. We choose to show the results for a time smaller than $T_0$ and a time bigger than $T_0$, so both behaviors can be observed. The simulations have been performed using an ad hoc algorithm exploiting the results of Theorems \ref{theo:TleqT0} and \ref{theo:TgeqT0}.

In Figure \ref{fig:K_sin} we can see the results for $K(\cdot)=K_S(\cdot)$. This choice models a scenario where mosquitoes are concentrated in the center of the domain studied and their concentration fades out as we move towards the boundaries. In the case $C=30$ (left column), we observe how the optimal strategy flattens and widens as $T$ increases. To understand this effect, we recall that in case $p_0(x)<\theta$, $p(t,x)$ is decreasing with respect to $t$. In other words, the \textit{Wolbachia}-infected mosquitoes tend to be replaced by wild mosquitoes if they don't surpass a critical threshold. Furthermore, if $p_0(x)>\theta$, $p(t,x)$ is increasing with respect to $t$, and therefore \textit{Wolbachia}-infected mosquitoes take over the population without further intervention. For the parameters considered we have $\theta\approx 0.21$ (The green dash-dotted line in Figure \ref{fig:K_sin}). According to our interpretation, since for $T$ small, $p(T,x)$ is close to its initial condition, a value of $p_0(x)$ below the threshold does not impact greatly the final result and having a bigger initial proportion in places where $K(x)$ is higher is favored. On the other hand, when $T$ is big, $p(T,x)$ can be far from its initial condition. Hence, there is an incentive for $p_0^*(x)$ to be above $\theta$, since the proportion of \textit{Wolbachia}-infected mosquitoes in the parts of the domain above $\theta$ will naturally increase with time. This also explains why the end of the release interval is abrupt. If the release is not going to achieve the critical proportion, it is better not to release. This effect can be seen clearly in the bottom-left graph and should be more pronounced the bigger $T$ is.

The changes between $T=1$ and $T=25$ in the case $C=200$ are imperceptible. A possible explanation is that in this case $p_0^*(x)>\theta$ for all $x\in\Omega$ already in the first case. Therefore the proportion of \textit{Wolbachia}-infected mosquitoes is going to increase with time everywhere. This case illustrates how $u^*_0(\cdot)$  is non-decreasing as $K(\cdot)$ increases, and how that is not necessarily the case of $p_0^*(\cdot)$, which decreases when $K(\cdot)$ increases if $p_0^*(\cdot)=p_M(\cdot)$. We recall that we have proven this monotonicity property for the case $T \leq T_0$ and in the case $T>T_0$ but $|\tilde{\Omega}_{\lambda_0}|=0$, which is the case in the bottom-right graph ($C=200$, $T=25$). The only case where  $|\tilde{\Omega}_{\lambda_0}|>0$ in Figure \ref{fig:K_sin} is the one with $C=30$, $T=25$ (bottom-left graph).

\begin{figure}[htb!]
	\centering
	\includegraphics[width=0.9\textwidth]{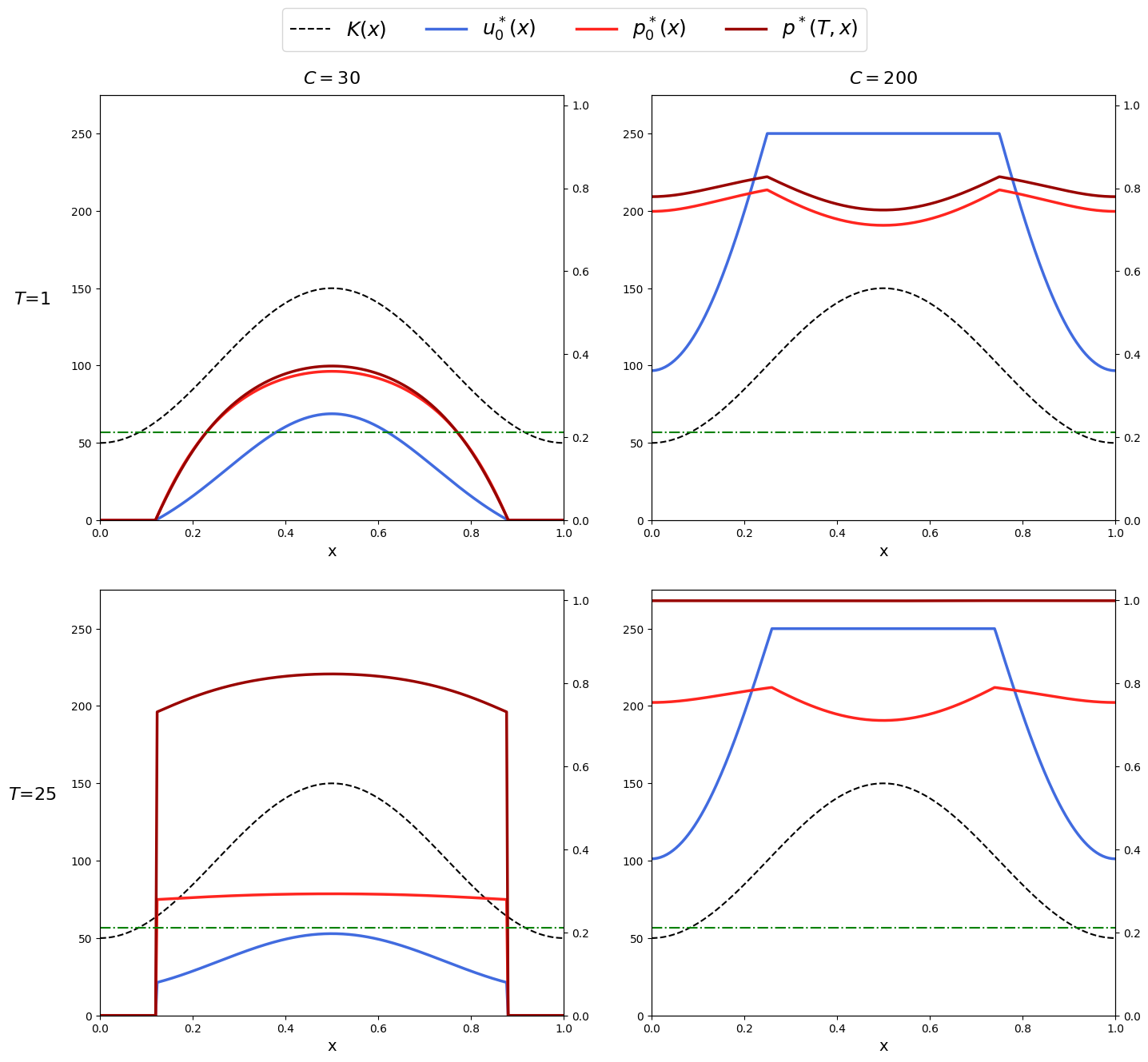}
	\caption{Solutions to problems \eqref{prob:u0}, $u_0^*(x)$, and \eqref{prob:p0}, $p_0^*(x)$, with a continuous carrying capacity $K(x):=K_S(x)$ as defined in \eqref{K_simulations}, in $\Omega=[0,1]$, for different time horizons and amounts of available mosquitoes. $p^*(T,x)$ stands for the solution of equation \eqref{eq:psimpl} with initial data $p_0^*(x)$. The values of $K(x)$ and $u_0^*(x)$ are read on the left axis, the values of $p_0^*(x)$ and $p^*(T,x)$ are read on the right axis. In green, the threshold $p=\theta$.}
	\label{fig:K_sin}
\end{figure}

In Figure \ref{fig:K_patch} we show the results for the simulations with $K(\cdot)=K_P(\cdot)$. This figure represents a scenario with two patches of land with two very distinct conditions for mosquitoes, as it can be the case, for example, of an urban area close to a wetland. In the case $C=30$ we can observe again the difference between the short-term and long-term strategies. With $T=1$ reaching a higher proportion on the left patch is favored, leaving the second patch untreated. Meanwhile, when the time horizon is increased, it also increases the incentive to release in a wider area above the critical proportion $p_0(x)=\theta$. Therefore the optimal releasing strategy consists in releasing a slightly smaller amount on the left patch in order to release in a certain domain in the right one.

In this case, on the bottom-left graph, we are in the case where $|\tilde{\Omega}_{\lambda_0}|>0$, and thus the secondary problem must be solved for the values of $\lambda$ in $[\lambda_0,\lambda_1]$ (see \eqref{prob:lambda}). Since $K(\cdot)$ is simple, the amount of mosquitoes released and the size of each subdomain of the right patch can be determined almost explicitly. For a fixed value of $\lambda\in[\lambda_0,\lambda_1]$, the amount of mosquitoes released in the left patch is $$\int_0^{\frac{|\Omega|}{2}}K(x)G\lrp{p_0^*(x)}\,dx=\int_0^{\frac{|\Omega|}{2}}\frac{3K_0}{2}G\lrp{\psi^1_{x,T}\lrp{\frac{2\lambda}{3K_0}}}\,dx=\frac{3|\Omega|K_0}{4}G\lrp{\psi^1_{x,T}\lrp{\frac{2\lambda}{3K_0}}},$$ analogously in the right patch $p_0^*(x)=\psi^1_{x,T}\lrp{\frac{2\lambda}{K_0}}$ wherever it is not $0$. Therefore the size of the subdomain $D$,  where mosquitoes are released in the right patch is such that $$C=\frac{3|\Omega|K_0}{4}G\lrp{\psi^1_{x,T}\lrp{\frac{2\lambda}{3K_0}}}+\frac{|D|K_0}{2}G\lrp{\psi^1_{x,T}\lrp{\frac{2\lambda}{K_0}}}.$$ This equality can be satisfied for different values of $\lambda$ and $|D|$. To find the optimal value $\lambda$ and thus the optimal value of $|D|$, we solve the one-dimensional optimization problem \eqref{prob:lambda}, which in this case reads
\begin{eqnarray*}
&& \inf_{\lambda\in[\lambda_0,\lambda_1]} \frac{3|\Omega|K_0^2}{16}((1-p^l(T))^2)+\frac{|D|K_0^2}{4}((1-p^r(T))^2) \\
&& =  \inf_{\lambda\in[\lambda_0,\lambda_1]} \frac{3|\Omega|}{4}((1-p^l(T))^2)+|D|((1-p^r(T))^2 \\
&& = \inf_{\lambda\in[\lambda_0,\lambda_1]} \frac{3|\Omega|}{4}((1-p^l(T))^2)+\frac{1}{G(p_0^r)}\lrp{\frac{2C}{K_0}-\frac{3|\Omega|}{2}G(p_0^l)}((1-p^r(T))^2
\end{eqnarray*}
where $p^l(T)$ and $p^r(T)$ solve equation $p'(t)=f(p(t))$ with initial condition $p_0^l=\psi^1_{x,T}\lrp{\frac{2\lambda}{3K_0}}$ and $p_0^r=\psi^1_{x,T}\lrp{\frac{2\lambda}{K_0}}$ respectively. This solution, nonetheless, is what we have been calling unique `up to a rearrangement'. As long as the size of the domain where mosquitoes are released is preserved, the solution can be moved on the right half of the domain and still be optimal. In Figure \ref{fig:K_patch_reverse} we show another choice for the solution.

Once again, in the case $C=200$ (right column) the solution does not change significantly when $T$ is increased. We can see how the monotonicity of $u_0^*(\cdot)$ with respect to $K(\cdot)$ is respected, but not for $p_0^*(\cdot)$. In this case, releasing a smaller amount of mosquitoes in the right patch induces a higher initial proportion due to the smaller carrying capacity there.

\begin{figure}[htb!]
	\centering
	\includegraphics[width=0.9\textwidth]{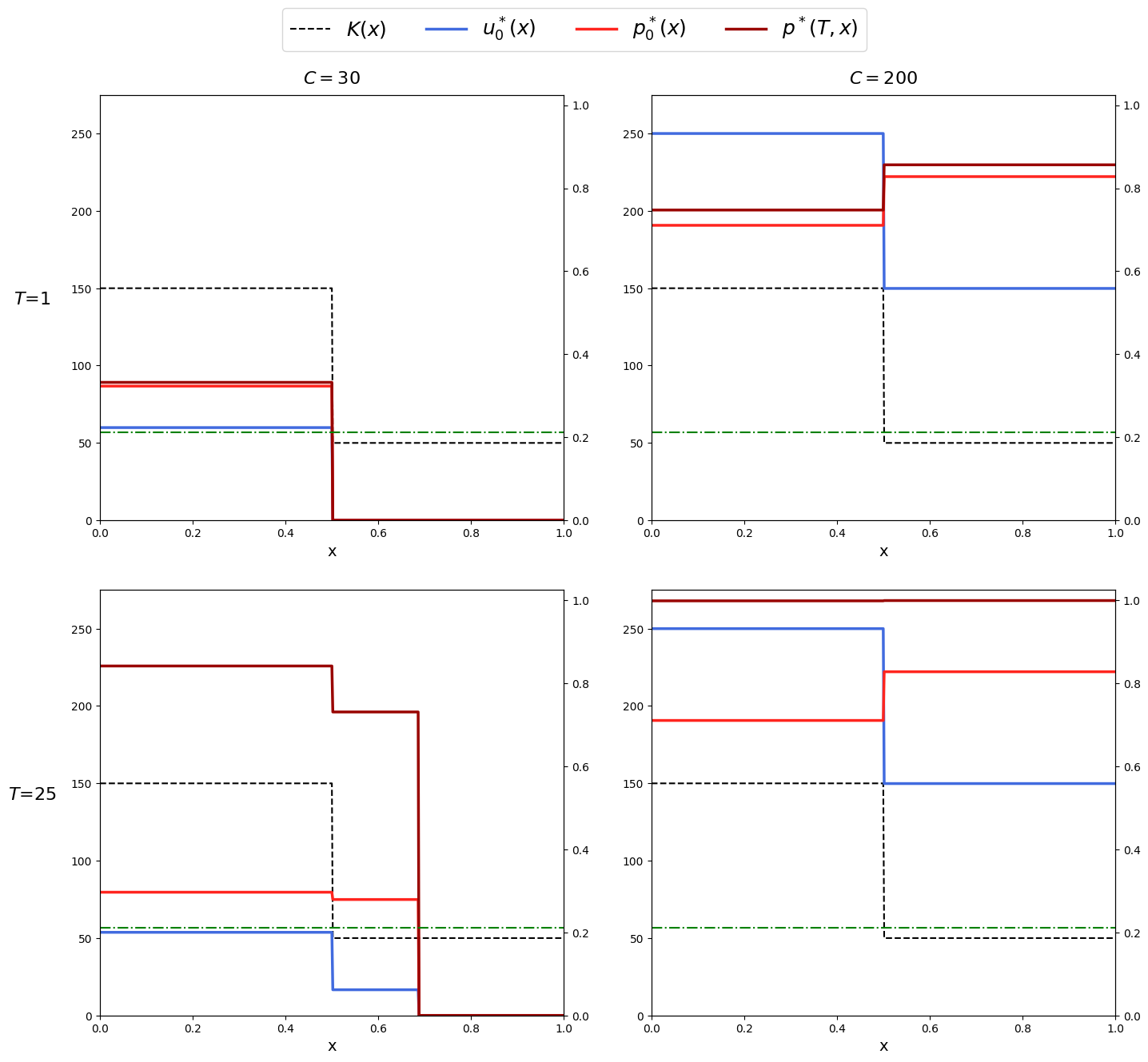}
	\caption{Solutions to problems \eqref{prob:u0}, $u_0^*(x)$, and \eqref{prob:p0}, $p_0^*(x)$, with a piece-wise constant carrying capacity $K(x):=K_P(x)$ as defined in \eqref{K_simulations}, in $\Omega=[0,1]$, for different time horizons and amounts of available mosquitoes. $p^*(T,x)$ stands for the solution of equation \eqref{eq:psimpl} with initial data $p_0^*(x)$. The values of $K(x)$ and $u_0^*(x)$ are read on the left axis, the values of $p_0^*(x)$ and $p^*(T,x)$ are read on the right axis. In green, the threshold $\theta$.}
	\label{fig:K_patch}
\end{figure}

\begin{figure}[htb!]
	\centering
	\includegraphics[width=0.57\textwidth]{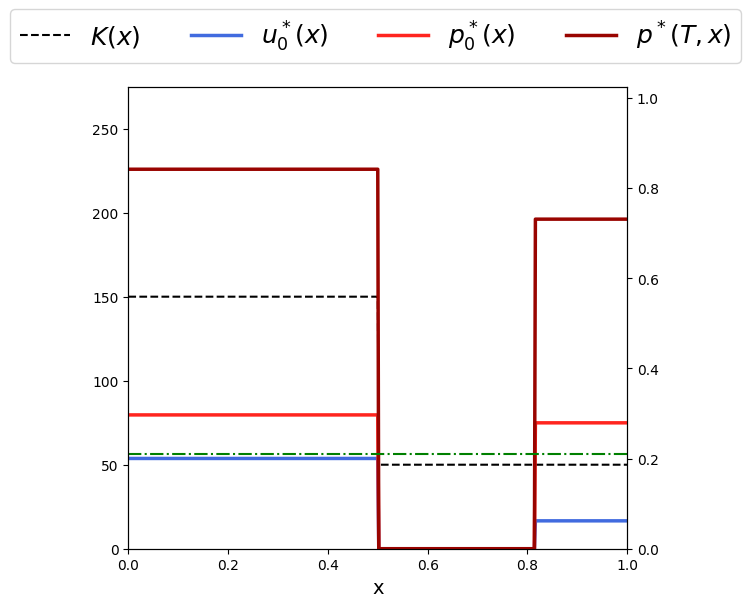}
	\caption{Alternative arrangement of solutions to problems \eqref{prob:u0}, $u_0^*(x)$, and \eqref{prob:p0}, $p_0^*(x)$, with a piece-wise constant carrying capacity $K(x):=K_P(x)$ as defined in \eqref{K_simulations}, in $\Omega=[0,1]$, with $C=30$ and $T=25$. $p^*(T,x)$ stands for the solution of equation \eqref{eq:psimpl} with initial data $p_0^*(x)$. The values of $K(x)$ and $u_0^*(x)$ are read on the left axis, the values of $p_0^*(x)$ and $p^*(T,x)$ are read on the right axis. In green, the threshold $\theta$.}
	\label{fig:K_patch_reverse}
\end{figure}

\subsection{2D simulations}

The method developed in this work can be applied in any dimension. Problems \eqref{prob:u0} and \eqref{prob:p0} can be solved using Theorems \ref{theo:TleqT0} and \ref{theo:TgeqT0} or, at least, reduced to the one-dimensional optimization problem \eqref{prob:lambda}, independently of the number of spatial dimensions considered. For a real application, nonetheless, the most interesting case is 2D.

Despite not presenting any novelties conceptually speaking, to illustrate the potential of the results, we show also two more simulations done in a 2D setting. We considered the following carrying capacity 
\begin{equation}\label{K_2D_simulations}
K_{2D}(x,y)=K_0\lrp{1-\frac{1}{6}\cos\lrp{\frac{2\pi x}{L_x}}-\frac{1}{3}\cos\lrp{\frac{2\pi y}{L_y}}}.
\end{equation}
As in the case with $K_S$, $K_{2D}$ models a scenario with a higher concentration of mosquitoes towards the center of the domain and a smaller one towards the boundaries. Nevertheless, note that $K_{2D}$ is not radially symmetric.

For the simulations we took $\Omega=[0,L_x]\times[0,L_y]$, with $L_x=L_y=1$. Once again, $\int_{\Omega}K_{2D}(x,y)\,dx\,dy=K_0|\Omega|$ for the chosen parameters. The results of the simulations can be seen in Figure \ref{fig:2D_K_sin}. We portray only the case $T=25$ for two values of $C$. Results match the intuition one can have from the related 1D case $K(\cdot)=K_{S}(\cdot)$. When a small amount of mosquitoes is considered, $C=30$ the solution is flat and wide to surpass the critical proportion $p_0=\theta$ in a bigger area, since the proportion of \textit{Wolbachia}-infected mosquitoes will naturally increase in those places. Also, $u_0^*(x)=0$ outside of this area for the reasons already exposed. On the other hand when a bigger amount of mosquitoes is considered, $C=200$, the solution is bigger than $p_0=\theta$ everywhere and varies more rapidly, being higher where the carrying capacity is higher, but flattening out when $u_0^*=M$ is reached.
 
\begin{figure}[htb!]
	\centering
	\includegraphics[width=0.9\textwidth]{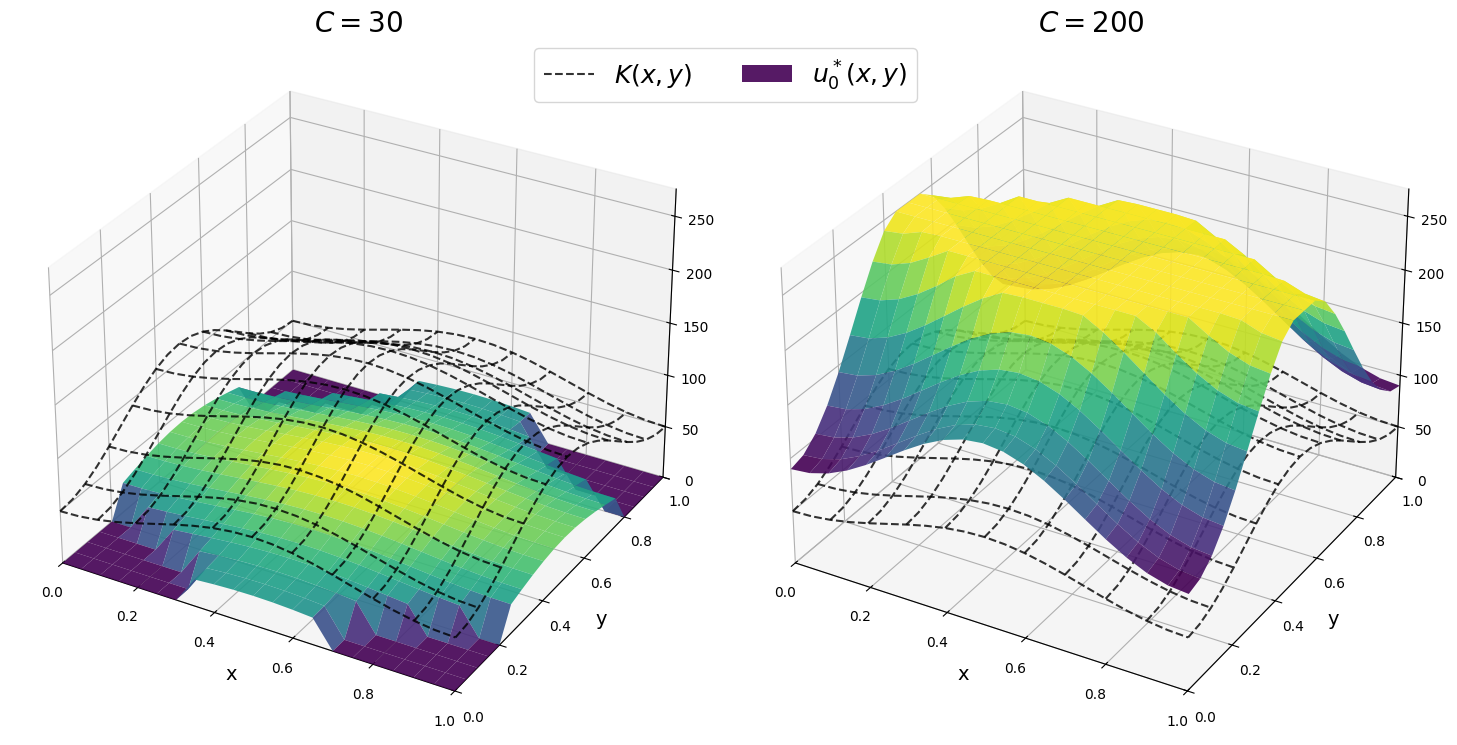}
	\caption{Solutions to Problem \eqref{prob:u0}, $u_0^*(x,y)$, in a two dimensional setting, with carrying capacity $K(x,y):=K_{2D}(x,y)$ as defined in \eqref{K_2D_simulations}, in $\Omega=[0,1]^2$, for different amounts of available mosquitoes and a time horizon of $T=25$.}
	\label{fig:2D_K_sin}
\end{figure}

\section{Discussion}\label{sec:Discussion}

In this paper, we examine the influence of spatial heterogeneity on an optimal release of \textit{Wolbachia}-infected mosquitoes in the context of the population replacement technique.

The case without considering spatial variability has been extensively studied in \cite{APSV}. This paper generalizes those results by incorporating a spatial component, allowing the carrying capacity to vary across the domain, which is typically the case in natural environments. Additionally, we introduce a global constraint across the entire domain. Despite its limitations, our study is significant because, to the best of our knowledge, spatial inhomogeneity has not been adequately addressed in the literature, even though it plays a crucial role in shaping optimal mosquito releases.

Under hypothesis \eqref{H} on the biological parameters of the problem, we explicitly identify a threshold in the time horizon of the problem, $T_0$. For time horizons shorter than this threshold, we prove the existence and characterize optimal release profiles. A key implication of this first result is that the solution to Problem \eqref{prob:u0}, $u_0^*$, is non-decreasing where $K$ is increasing. This has a direct translation for practical purposes in the field: under the aforementioned circumstances and for short time horizons, releases should be more intense in areas with higher initial mosquito populations.

On the other hand, for larger time horizons, we either directly characterize the optimal solution or reduce \eqref{prob:u0} to a one-dimensional optimization problem that can be easily solved numerically. This reduction can be achieved regardless of the number of dimensions considered for the releases (typically 1D or 2D in practical scenarios). This result significantly simplifies the numerical construction of optimal release profiles for large time horizons.

We remark that although we do not establish existence of solutions for Problem \eqref{prob:u0} for a general carrying capacity function $K$, we were able to prove existence of a minimizer for specific forms of the function $K$~: We demonstrated this for constant $K$ (see Corollary \ref{coro:solKconstant}) and for piecewise constant $K$ in one dimension (see Appendix \ref{app:Existence}). The case where $K$ is piecewise constant is particularly relevant as, in practice, $K$ is estimated from field measurements taken at a finite number of points in space.

While our results are valid only under hypothesis \eqref{H}, we show numerically (see Appendix \ref{app:H}) that a significant portion of the parameter space satisfies this condition.

The primary limitation of this paper lies in the simplifying assumption of negligible mosquito mobility. While assuming that mosquitoes do not disperse simplifies the analysis, this comes at the expense of losing realism in the modeling. However, this simplification allows us to study an ordinary differential equation, Problem \eqref{prob:u0}, instead of a much harder to tackle partial differential equation, Problem \eqref{prob:full}. \textit{Aedes} mosquitoes are known to have limited flight capabilities, with some mark-release-recapture experiments showing that 90\% of recaptured mosquitoes remain within a 200m radius of the release point \cite{migration_rate}. Therefore, if mosquito mobility is sufficiently low, the results presented here can provide meaningful insights into real-world scenarios, as long as solutions of Problem \eqref{prob:reduced} converge to those of Problem \eqref{prob:u0} as $D\to 0$.

Comparing solutions to Problem \eqref{prob:reduced} for small values of $D$, with solutions to Problem \eqref{prob:u0} for the case $D=0$ (see Figure \ref{fig:Limit_Diff_T1}) we observe that as $D$ decreases, solutions appear to converge. Solutions for $D>0$ have been obtained using the software GEKKO (see \cite{GEKKO}), while solutions for $D=0$ have been obtained  using our ad-hoc algorithm developed exploiting theorems \ref{theo:TleqT0} and \ref{theo:TgeqT0}. At $D=5\times 10^{-5}$, the solution closely resembles the $D=0$ case, with minor perturbations only at the domain edges and transition points. The formal analysis of the problem is left for future works.

\begin{figure}[htb!]
	\centering
	\includegraphics[width=\textwidth]{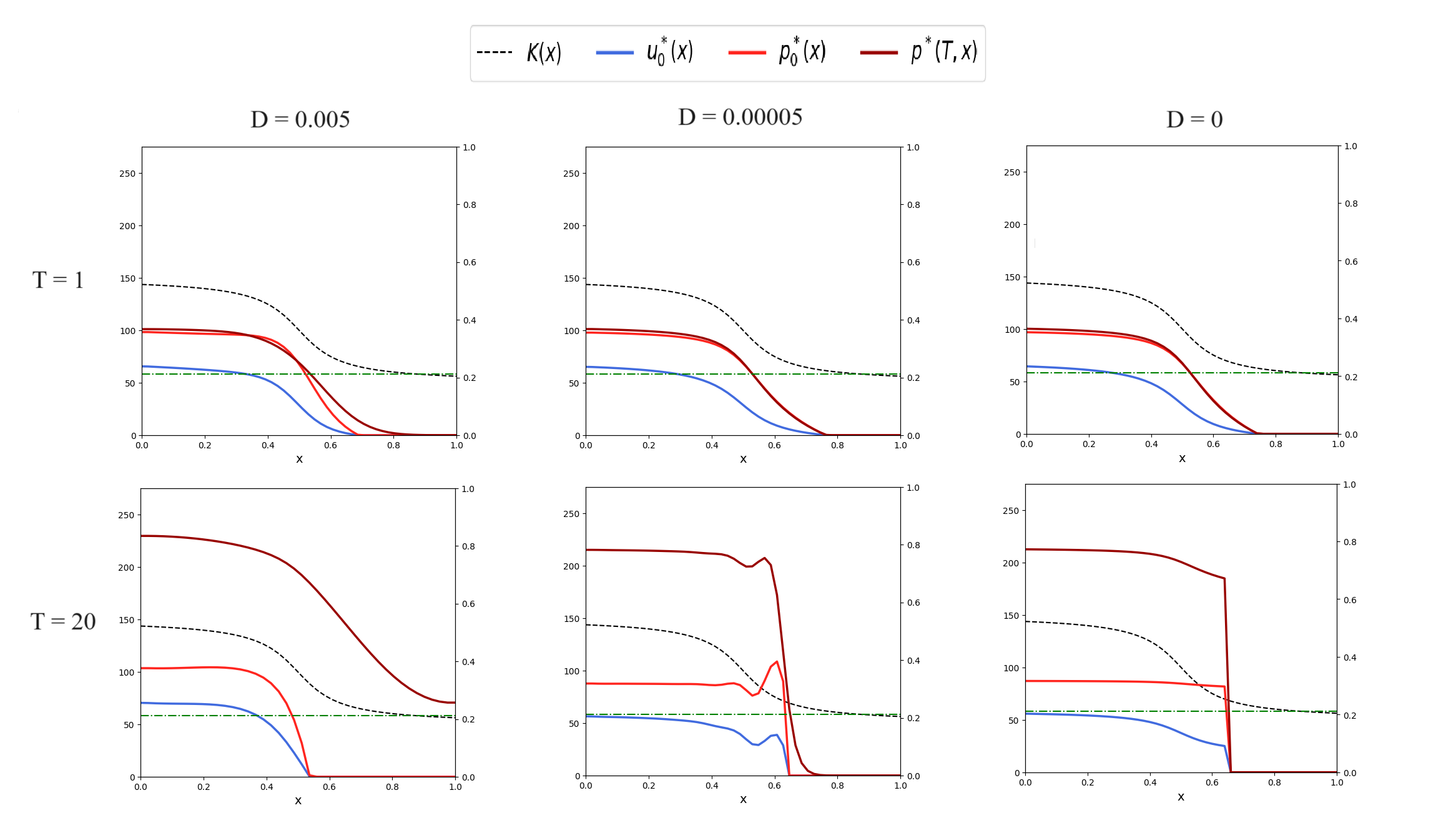}
	\caption{Solutions to problems \eqref{prob:u0}, $u_0^*$, and \eqref{prob:p0}, $p_0^*$, with $K(x)=K_0\lrb{1+\frac{1}{\pi}\arctan\lrp{-10\lrp{x-\frac{|\Omega|}{2}}}}$ modelling an uneven distribution of mosquitoes with a smooth transition, for different time horizons ($T\leq T_0$ and $T > T_0$), $C=30$ and a decreasing diffusion rate from left to right. Absolute amounts must be read in the left axis, while proportions must be read in the right axis. $p^*$ stands for the solution of equations \eqref{eq:p}, cases $D>0$, and \eqref{eq:psimpl}, case $D=0$, with initial data $p_0^*(\cdot)$. In green, the threshold $p=\theta$.}
	\label{fig:Limit_Diff_T1}
\end{figure}

In conclusion, numerical studies suggest that, at least in certain scenarios, solutions converge to the ones we obtained as the diffusion rate approaches zero. This indicates that despite the simplifications made, if mosquito mobility is sufficiently low —an observation supported by certain field releases \cite{Hoffmannetal, Deployment_Townsville}— our results can approximate the optimal distribution for a single release at initial time of \textit{Wolbachia}-infected mosquitoes.

\section*{Acknowledgement}

The authors would like to aknowledge warmly Pr. Yannick Privat for fruitful discussions and useful comments and remarks. 
\vspace{0.2cm} \\
\begin{minipage}{2cm}
\includegraphics[width=\textwidth]{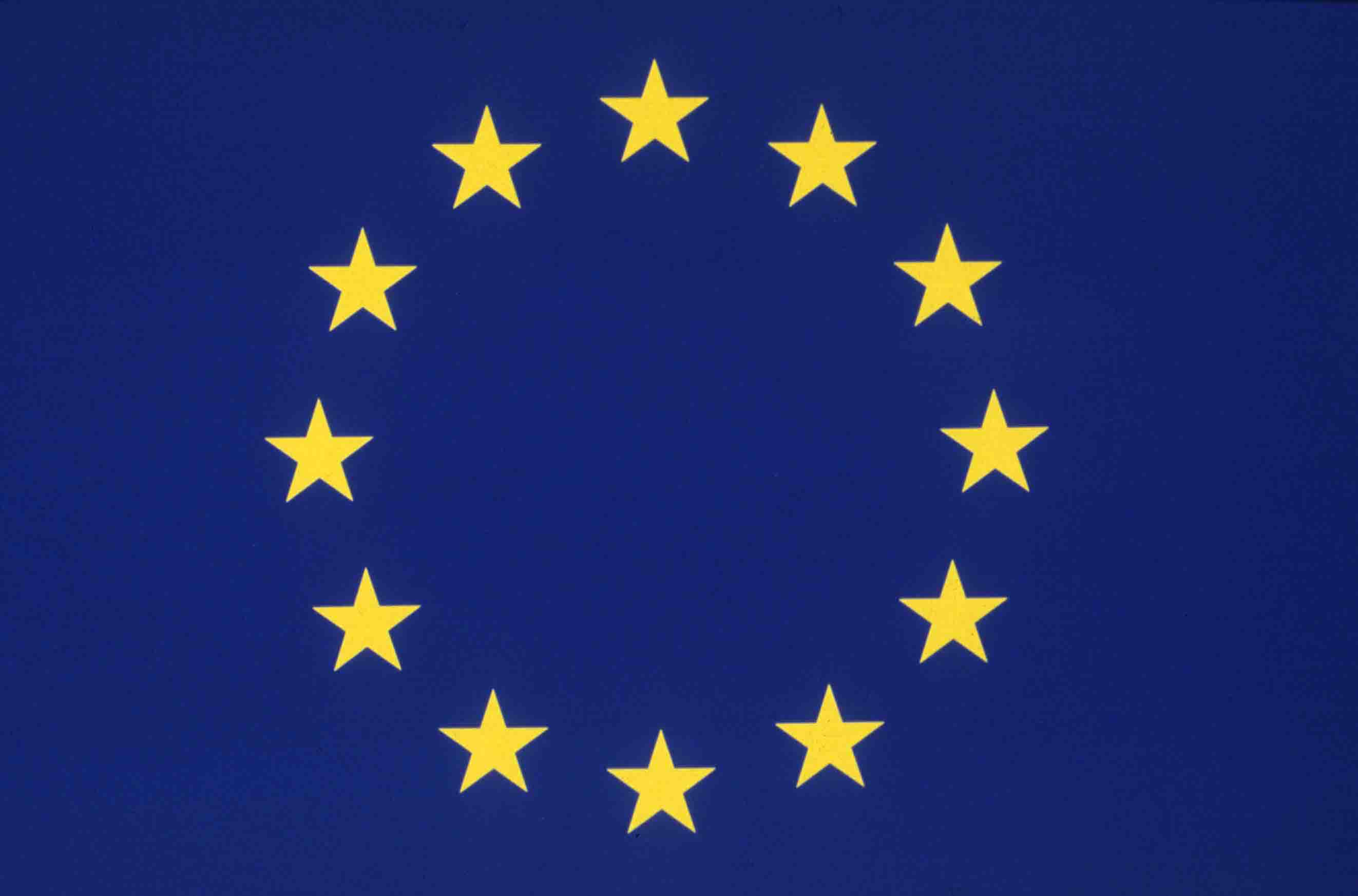}
\end{minipage}
\hspace{0.2cm}
\begin{minipage}{14.25cm}
	 This project has received funding from the European Union’s Horizon 2020 research and innovation programme under the Marie Skłodowska-Curie grant agreement No 754362.
\end{minipage}\ \\

\hspace{0.5cm}
\appendix
\begin{center}
	\fbox{\Large{\textsf{Appendix}}}
\end{center}

\appendix

\section{Numerical exploration of the parameter space for Hypothesis \eqref{H}}\label{app:H}

\begin{proposition}\label{prop:f2zero}
	Let us consider $$f(p) := b_1^0 d_2 s_h \frac{p(1-p)(p-\theta)}{b_1^0 (1-p) (1-s_h p)+ b_2^0 p}\quad \text{and}\quad\theta := \frac{1}{s_h}\lrp{1-\frac{d_1 b_2^0}{d_2 b_1^0}}.$$ Assuming $d_1\leq d_2 \leq b_2^0 \leq b_1^0$ and $0<\theta<1$, then $f''$ admits a single zero in $(0,1)$.
\end{proposition}

\begin{proof}
	The existence of a zero of $f''$ is straightforward to prove: we have $f(0)=f(\theta)=f(1)=0$, thus thanks to Rolle's theorem, there exist two zeros of $f^{\prime}$ in $(0,1)$. Then, applying again Rolle's theorem, we prove the existence of some zero $\theta_2$ of $f^{\prime\prime}$, lying in between the two zeros of $f^{\prime}$. Thus, $\theta_{2}\in(0,1)$ in all generality.
	
	Let us now prove the uniqueness of the zero of $f^{\prime\prime}$ in $[0,1]$. By computing the rational function $f^{\prime\prime}$,
	we see that $\{f^{\prime\prime}=0\}=\{R=0\}$ where, denoting $\kappa:=1+s_{h}-\frac{b_{2}^0}{b_{1}^0}>0$,
	\begin{align*}
	R(p) & :=(s_{h}-s_{h}^{2}\theta-\kappa^{2}+\kappa s_{h}+\kappa s_{h}\theta)p^{3}+3(\kappa-s_{h}-s_{h}\theta)p^{2}-3(1-s_{h}\theta)p+\theta-\kappa\theta+1\\
	& =\bar{A}p^{3}+\bar{B}p^{2}+\bar{C}p+\bar{D}
	\end{align*}
	Thus to conclude, it is enough to prove that $R$ has a unique zero
	in $[0,1]$.
	Because $d_1\leq d_2\leq b_2^0\leq b_1^0$ holds, we find that 
	\[
	\bar{A}>0,\qquad \bar{B}\leq 0,\qquad \bar{C}<0,\qquad \bar{D}>0.
	\]
	Then, using Descarte's rule of signs, we find that $R$ has zero or
	two positive roots. However, since $f^{\prime\prime}$ has at least
	one zero in $(0,1)$, so does $R$, so that $R$ admits exactly two
	positive roots. Meanwhile, applying the rule to $p\mapsto R(-p)$
	implies that $R$ has one negative root. 
	
	Now, we set $S(p)=R(p+1)$. In particular, the leading coefficient
	of $S$ is $\bar{A}>0$ while one can prove that
	\[
	S(0)=\bar{A}+\bar{B}+\bar{C}+\bar{D}<0.
	\]
	Clearly $S$ has three real roots, and their product is given by $-\frac{S(0)}{\bar{A}}>0$.
	However, $S$ has at least one negative root since $R$ also does. Since
	the product of all three roots of $S$ is positive, $S$ has exactly
	two negative roots and one positive root. As a result, $R$ has exactly
	one root in $[0,1]$, and the conclusion follows.
\end{proof}

This section is devoted to a numerical exploration of the space of parameters, in order to establish the validity of Hypothesis \eqref{H}. For the exploration, we normalize the parameters assuming $b_1^0=1$. The results are presented in Figures \ref{fig:HA_4x4} and \ref{fig:HA_2x2}. To produce these images for a given pair of values $(s_h,b_2)$, two random values for $d_2$ and $d_1$ are chosen, such that $d_1\leq d_2\leq b_2^0\leq b_1^0=1$. If the randomly generated set of parameters are such that $\theta\not\in(0,1)$, the set is discarded. If indeed $0<\theta<1$, then the Hypothesis \eqref{H} is tested. In blue, are the values of the parameters for which Hypothesis \eqref{H} is satisfied. In red, the values for which is not.
\begin{figure}[h!]
	\centering
	\includegraphics[width=\textwidth]{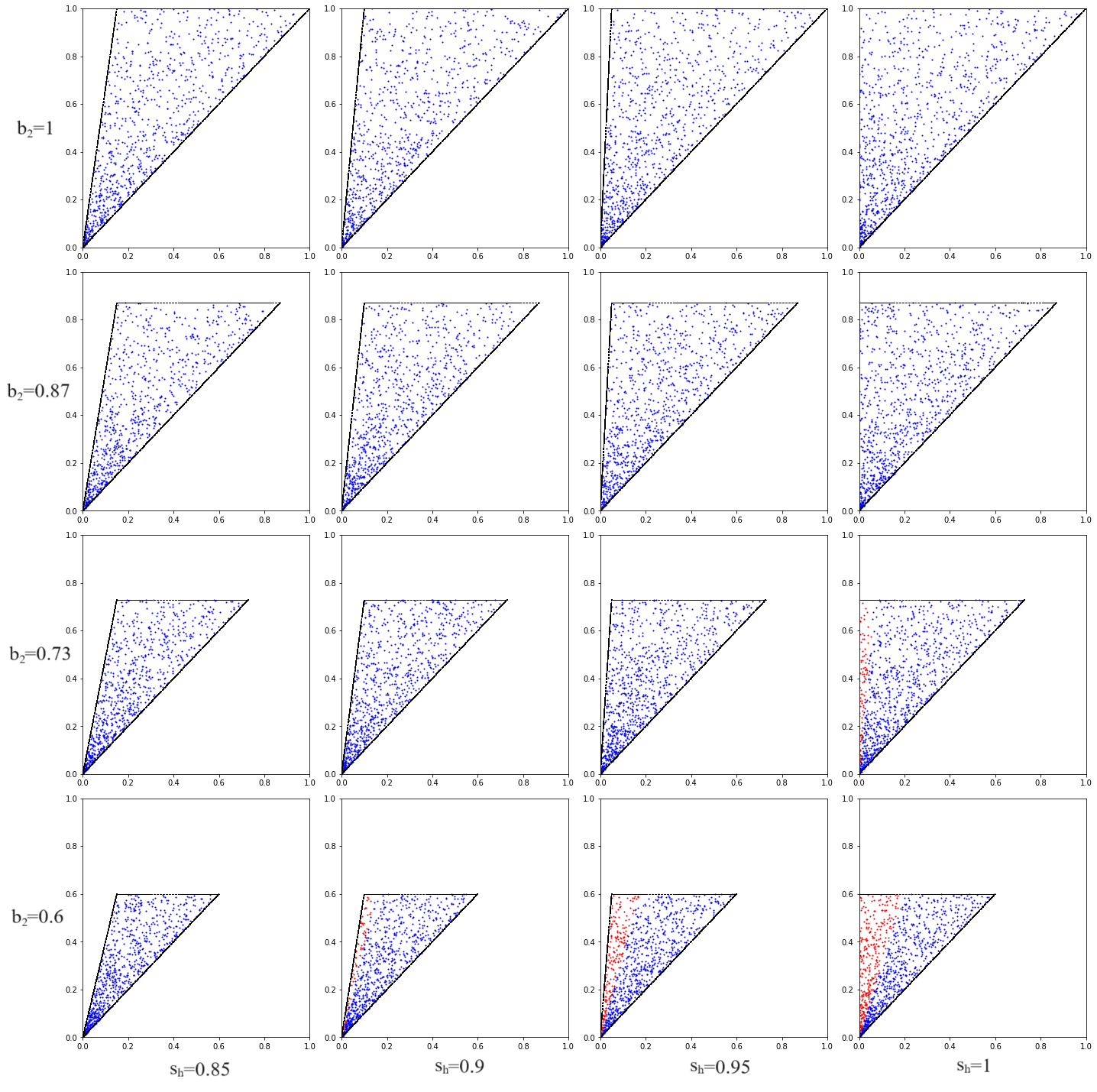}
	\caption{Hypothesis \eqref{H} tested for different parameter values. All pictures have $d_1$ in the $x$-axis and $d_2$ in the $y$-axis. In this image, $s_h\in\left\{0.5,0.67,0.83,1\right\}$ increases from left to right and $b_2\in\left\{0.6,0.73,0.87,1\right\}$ decreases from top to bottom. Blue dots mean Hypothesis \eqref{H} is satisfied for those parameters, while red dots mean it is not.}
	\label{fig:HA_4x4}
\end{figure}
\begin{figure}[h!]
	\centering
	\includegraphics[width=0.6\textwidth]{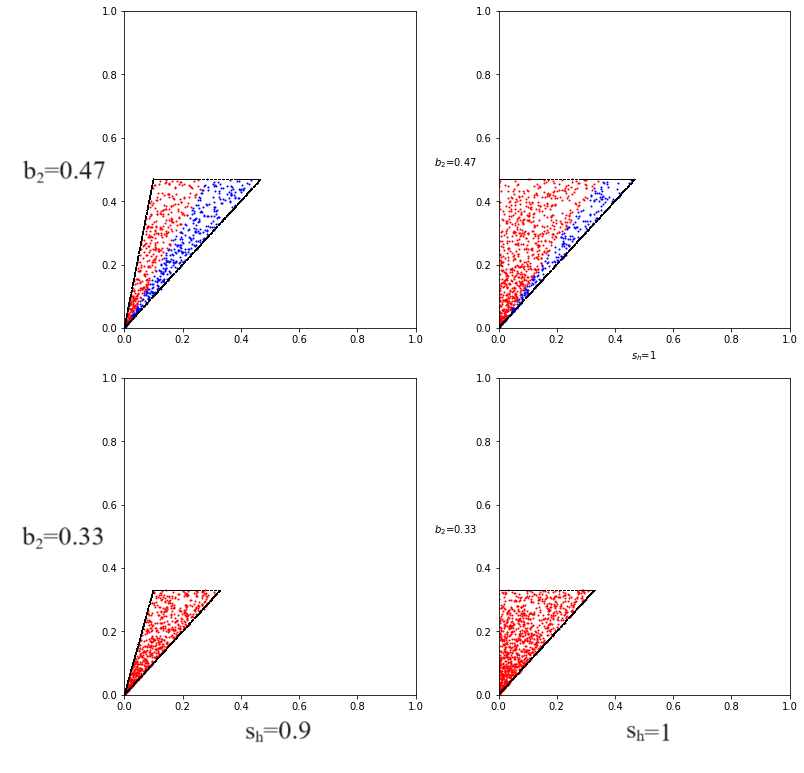}
	\caption{Hypothesis \eqref{H} tested for small values of $b_2$ and high values of $s_h$. All pictures have $d_1$ in the $x$-axis and $d_2$ in the $y$-axis. In this image, $s_h\in\left\{0.9,1\right\}$ increases from left to right and $b_2^0\in\left\{0.33,0.47\right\}$ decreases from top to bottom. Blue dots mean Hypothesis \eqref{H} is satisfied for those parameters, while red dots mean it is not.}
	\label{fig:HA_2x2}
\end{figure}

 The condition $\theta\in(0,1)$ can be written in terms of the other parameters,$$ 0<\theta<1 \Leftrightarrow 0<1-\frac{d_1b_2^0}{d_2b_1^0}<s_h \Leftrightarrow 1>\frac{d_1b_2^0}{d_2b_1^0}>1-s_h \Leftrightarrow \frac{d_2}{d_1}>\frac{b_2^0}{b_1^0}>\frac{d_2}{d_1}(1-s_h).$$ Since $b_2^0\leq b_1^0$ and $d_1\leq d_2$, this means that if $b_2^0\leq 1-s_h$, necessarily $\theta>1$, thus these values can be excluded from the exploration. Indeed, the black lines wrapping the dots in Figures \ref{fig:HA_4x4} and \ref{fig:HA_2x2} are $d_2=d_1$, $b_2^0=1$ and $b_2^0=1-s_h$.

 As we can see in Figures \ref{fig:HA_4x4} and \ref{fig:HA_2x2}, Hypothesis \eqref{H} is satisfied by most of the parameters of the parameter space. For high values of $b_2^0$ it is always satisfied (also for those values not shown in Figure \ref{fig:HA_4x4}). As $b_2^0$ decreases, red dots appear only when we have high values of $s_h$ and $d_2\gg d_1$. Only for high values of $s_h$ and small values of $b_2^0$ the red dots dominate the picture. In a realistic scenario, based on the values for the parameters found in the literature \cite{APSV,DHP} we expect a high value of $s_h$ and $b_2^0$ and smaller values of $d_2$ and $d_1$, for which hypothesis \eqref{H} is satisfied. The values of the parameters not satisfying Hypothesis \eqref{H} would represent a particular strain of \textit{Wolbachia} that, in a certain variety of mosquito, would produce a high CI rate and a big penalty on their fertility, which, a priori, is not impossible.

\section{Comments on the existence of solutions for Problem \eqref{prob:p0}}\label{app:Existence}

In this work, we do not establish the existence of minimizers for Problem \eqref{prob:p0} in all generality. Despite not exploring the matter separately, we do prove the existence of solutions in the case $T\leq T_0$ in Theorem \ref{theo:TleqT0}, under hypothesis \eqref{H}. Indeed, in this theorem, we provide necessary conditions that solutions must satisfy. Then, in its proof, we exploit these conditions to narrow down the solution space, obtaining a single candidate to solution, $p_0^*$. Since the constructed $p_0^*$ is the only candidate solution and we prove that $p_0^*\in\mathcal{P}_{0,C,M}$, we conclude not only the existence but also the uniqueness of the solution (as already mentioned, up to a rearrangement in case $K(\cdot)$ is piecewise constant). Conversely, for the case $T>T_0$, always under hypothesis \eqref{H}, the candidate solution we construct in the proof of Theorem \ref{theo:TgeqT0} by exploiting the necessary optimality conditions is not unique, but rather a whole family of candidates depending on a parameter $\lambda^*\in[\lambda_0,\lambda_1]$. Hence we cannot conclude the existence of solutions in general from this result.

Although tackling comprehensively the existence of minimizers is beyond the scope of this work, we present here a partial result settling the existence in 1D, for all $T>0$, for the case where $K(\cdot)$ is piecewise constant. This proof, nevertheless, cannot be easily extended to the general case.

\begin{proposition}
	Let $K$ be a piece-wise constant function, with $x\in \Omega\subset\mathbb{R}$. Then, there exists $p_0^*\in\mathcal{P}_{0,C,M}$ solving Problem \eqref{prob:p0}.
\end{proposition}

\begin{proof}
	Let us write $K(x)=\sum_{i=1}^n K_i \mathbbm{1}_{[x_{i-1},x_i]}$. We place ourselves in one of the intervals $[x_{i-1},x_i]$, where $K(\cdot)$ is constant.
	
	Let us consider any function $p_0\in\mathcal{P}_{0,C,M}$ in this interval. We claim there exists a monotonic (decreasing or increasing) rearrangement of $p_0$, that we will denote $\widehat{p}_0$ such that $\widehat{p}_0\in\mathcal{P}_{0,C,M}$. To define this rearrangement, let us introduce, in a given interval $[x_{i-1},x_i]$, $$\mu(s):=\left|\{x\in[x_{i-1},x_{i}] \ : \  p_0(x)>s\}\right|.$$ Then, we define $\widehat{p}_0$ in that interval as $$\widehat{p}_0(x):=\inf\{s\in[0,1] \ : \  \mu(s)\leq x\}.$$ The fact that such a rearrangement will respect $0\leq \widehat{p}_0\leq G^{-1}\lrp{M/K(x)}$ is trivial since rearranging a function does not change its maximums or minimums (see \cite{Rearrangement}). Suppose $\int_\Omega K(x)G(p_0(x))\,dx\leq C$, then $$\int_\Omega K(x)G(\widehat{p}_0(x))\,dx=\int_{\Omega\setminus [x_{i-1},x_i]} K(x)G(p_0(x))\,dx+K_i\int_{x_{i-1}}^{x_i}G(\widehat{p}_0(x))\,dx$$
	for every $i\in \{ 1,\dots,n\}$.
	$G$ is a continuous function, hence it is measurable. Therefore, since $G$ is non-negative and measurable we have $$\int_{x_{i-1}}^{x_i}G(p_0(x))\,dx=\int_{x_{i-1}}^{x_i}G(\widehat{p}_0(x))\,dx$$
	by equimeasurability of the rearrangement.
	
	This implies that if $\int_\Omega K(x)G(p_0(x))\,dx\leq C$, then $\int_\Omega K(x)G(\widehat{p}_0(x))\,dx\leq C$. This reasoning can be easily extended to all of the subintervals. Therefore we have proved that $\mathcal{P}_{0,C,M}$  is stable under rearrangements.
	
	Let us define now $$\widehat{\mathcal{P}}_{0,C,M}:=\left\{\widehat{p}_0\in\mathcal{P}_{0,C,M} \ | \ \widehat{p}_0 \mbox{ is monotonic in } [x_{i-1},x_i], i=1,\dots,n  \right\}.$$ Observe that $$\inf_{p_0\in\mathcal{P}_{0,C,M}}J^0(p_0)=\inf_{\widehat{p}_0\in\widehat{\mathcal{P}}_{0,C,M}}J^0(\widehat{p}_0),$$ with $J^0$ defined by \eqref{J0}. Indeed, if $p_0^*$ is minimizer of $J^0$ in $\mathcal{P}_{0,C,M}$, then $$\int_\Omega K(x)^2(1-\widehat{p}(T,x))^2\,dx=\int_{\Omega\setminus [x_{i-1},x_i]} K(x)^2(1-p(T,x))^2\,dx+K_i^2\int_{x_{i-1}}^{x_i}(1-\widehat{p}(T,x))^2\,dx,$$ for every $i\in \{ 1,\dots,n\}$, where we are denoting by $\widehat{p}(t,x)$ the solution to equation \eqref{eq:psimpl} with initial condition $\widehat{p}_0(x)$.
	
	We can follow the same reasoning as before to prove that $$\int_{x_{i-1}}^{x_i}(1-p(T,x))^2\,dx=\int_{x_{i-1}}^{x_i}(1-\widehat{p}(T,x))^2\,dx$$
for every $i\in \{ 1,\dots,n\}$. To see this clearly, we can write $p(t,x)$ as a function of its initial condition by realising that $p(t,x)$ can be written as $$\frac{\partial}{\partial t}p(t,x)=f(p(t,x)) \Rightarrow \int_0^p \frac{d\nu}{f(\nu)}=\int_0^t ds=t.$$
	Defining $F(p)$ as the antiderivative of $1/f(p)$ vanishing at $0$, $F(p):=\int_0^p\frac{d\nu}{f(\nu)}$, we can write
	\begin{equation}\label{expr:pTp0}
	F(p(T,x))=F(p_0(x))+T \Rightarrow p(T,x)=F^{-1}\lrp{F(p_0(x))+T}.
	\end{equation}
	 Both $F$ and its inverse are continuous functions and thus, so it is its composition. Therefore, $p(T,\cdot)$ is also a measurable function of $p_0(x)$ and since it is non-negative both integrals are equal.

	This implies that, if there exists a solution monotonic by intervals, $\widehat{p}_0^{\,*}\in\widehat{\mathcal{P}}_{0,C,M},$ there must also exist a solution in $\mathcal{P}_{0,C,M}$. Thus, we restrict our analysis to the first kind of functions.
	
	Let us consider a minimizing sequence $(\widehat{p}_0^{\,n})_{n\in\mathbbm{N}}\in \widehat{\mathcal{P}}_{0,C,M}$ for Problem \eqref{prob:p0}. We know it exists since $\widehat{\mathcal{P}}_{0,C,M}$ is non-empty. Due to the fact that for all $n\in\mathbbm{N}$, $0\leq \widehat{p}_0^{\,n} (x)\leq G^{-1}(M/K(x))$ a.e. in $\Omega$ and using the monotonicity of $\widehat{p}_0^{\,n}$ on each interval $(x_{i-1},x_i)$, we deduce from Helly's selection theorem (see \cite{Selection_Helly}) that 
	%$\widehat{\mathcal{P}}_{0,C,M}$ is compact for the pointwise topology of $L^{\infty}(\Omega)$. It follows that 
	$(\widehat{p}_0^{\,n})_{n\in\mathbbm{N}}$ converges pointwisely to an element $\widehat{p}_0^{\,*}$, up to a subsequence. 

	 Basic properties of pointwise convergence lead us to conclude that $0\leq \widehat{p}_0^{\,*}(x) \leq G^{-1}(M/K(x))$ a.e. in $\Omega$. Moreover, according to the Lebesgue dominated convergence theorem, one has 
	 $$\int_{\Omega}K(x)G\lrp{\widehat{p}_0^{\,*}(x)}\,dx=\lim_{n\to\infty}\int_{\Omega}K(x)G\lrp{\widehat{p}_0^{\,n}(x)}\,dx=\lim_{n\to\infty}\left<K(x)G\lrp{\widehat{p}_0^{\,n}(x)},1\right>_{L^{\infty},L^1}\leq C.$$ 
	 Indeed, we recall that $K(\cdot)$ is piecewise constant and thus it does not affect the convergence properties of the sequence under the integral. Therefore, $\widehat{p}_0^{\,*}\in\widehat{\mathcal{P}}_{0,C,M}$. A similar reasoning shows that the sequence $(F^{-1}\lrp{F(\widehat{p}_0^{\,n}(x))+T})_{n\in \NN}$ converges almost everywhere in $\Omega$ and therefore, we have 
	 $$\lim_{n\to\infty}J^0\lrp{\widehat{p}_0^{\,n}(x)}=J^0\lrp{\widehat{p}_0^{\,*}(x)},$$ 
	 according to \eqref{expr:pTp0}.
	 It follows that $\widehat{p}_0^{\,*}$ is indeed a solution to Problem \eqref{prob:p0}.
\end{proof}

%%%%%%%%%%%%%%%%%%%%%%%%%%%%%%%%%%%
%
%%%%%% BIBLIO %%%%%%%%%%%%%%%%%%%%%%
%
%%%%%%%%%%%%%%%%%%%%%%%%%%%%%%%%%%%%

\bibliographystyle{abbrv}%{plain}
\bibliography{bib_inhomogeneity}

%Sorry, I changed the references to match what I had in the thesis so I don't have to translate them when I add the work there.

%MBE->ADPV
%SIMA->APSV
%COCV->DHP
%reduction->SV_Reduction

\end{document}